\newcommand\bbR{\mathbb{R}}
\newcommand\bbS{\mathbb{R}^{d\times d}_{\text{sym}}}

\newcommand\cC{\mathcal{C}}
\newcommand\cD{\mathcal{D}}
\newcommand\cF{\mathcal{F}}
\newcommand\cH{\mathcal{H}}
\newcommand\cP{\mathcal{P}}
\newcommand\cT{\mathcal{T}}

\newcommand\cX{\mathcal{X}}

\newcommand\cG{\mathcal{G}}

\newcommand\bn{\boldsymbol{n}}
\newcommand\bF{\boldsymbol{f}}
\newcommand\bu{\boldsymbol{u}}
\newcommand\bv{\boldsymbol{v}}
\newcommand\bw{\boldsymbol{w}}

\newcommand\be{\boldsymbol{e}}

\newcommand\bp{\boldsymbol{p}}
\newcommand\bq{\boldsymbol{q}}

\newcommand\beps{\boldsymbol{\varepsilon}}
\newcommand\bsig{\boldsymbol{\sigma}}
\newcommand\btau{\boldsymbol{\tau}}
\newcommand\bnabla{\boldsymbol{\nabla}}

\newcommand\bchi{\boldsymbol{\chi}}

\newcommand\mt{\mathtt{t}}

\newcommand{\tr}{\operatorname{tr}}
\renewcommand{\div}{\operatorname{div}}
\newcommand{\bdiv}{\operatorname{\mathbf{div}}}

\documentclass[11pt, leqno]{article} 
\usepackage[utf8]{inputenc} 
\usepackage[T1]{fontenc}
\usepackage[english]{babel} 

\usepackage{lmodern} 
\usepackage{stmaryrd} 

\usepackage{amsmath,amsthm,amsfonts,amssymb}

\usepackage{ninecolors}
\usepackage{graphicx}
\usepackage{tabularray}
\usepackage{caption, subcaption}
\usepackage{tikz}

\usepackage{xfrac} 

\usepackage{stmaryrd}
\newcommand{\jump}[1]{{\llbracket{#1}\rrbracket}}
\newcommand{\mmean}[1]{\left\{\kern-1.ex\left\{ #1 \right\}\kern-1.ex\right\}}

\usepackage{mathtools}
\DeclarePairedDelimiterX\norm[1]\lVert\rVert{
   \ifblank{#1}{\:\cdot\:}{#1}
}
\DeclarePairedDelimiterX\abs[1]\lvert\rvert{
   \ifblank{#1}{\:\cdot\:}{#1}
}
\DeclarePairedDelimiter{\inner}{(}{)}
\DeclarePairedDelimiter{\set}{\{}{\}}
\DeclarePairedDelimiter{\dual}{\langle}{\rangle}

\usepackage{siunitx}

\usepackage[singlespacing]{setspace} 

\usepackage{titlesec}
\titleformat{\section}[block]{\filcenter\bfseries}{\thesection.}{1em}{}
\titleformat{\subsection}[block]{\bfseries}{\thesubsection.}{1em}{}

\newtheoremstyle{paper}{}{}{\itshape}{}{\scshape}{.}{.5em}{}
\theoremstyle{paper}
\newtheorem{theorem}{Theorem}
\newtheorem{proposition}{Proposition}
\newtheorem{lemma}{Lemma}

\newtheorem{remark}{Remark}

\usepackage[numbers]{natbib}
\usepackage{natbib}
\setcitestyle{aysep={}}

\AtBeginEnvironment{thebibliography}{\setstretch{1.1}}

\usepackage{fancyhdr} 

\usepackage[pdftex, hidelinks, hypertexnames = false, hyperfootnotes = false,
pdfpagemode = UseNone, pdfdisplaydoctitle = true]{hyperref}%

\usepackage[a4paper, tmargin=3cm, bmargin=3cm, rmargin=2.2cm, lmargin=2.2cm]{geometry}

\begin{document}
\title{\huge Hybridizable Discontinuous Galerkin Methods for Coupled Poro-Viscoelastic and Thermo-Viscoelastic Systems}
\author{Salim Meddahi
	\thanks{This research was  supported by Ministerio de Ciencia e Innovación, Spain.}}

\date{}                       


\maketitle

\begin{abstract}

\noindent
This article presents a unified mathematical framework for modeling coupled poro-viscoelastic and thermo-viscoelastic phenomena, formulated as a system of first-order in time partial differential equations. The model describes the evolution of solid velocity, elastic and viscous stress tensors, and additional fields related to either fluid pressure or temperature, depending on the physical context. We develop a hybridizable discontinuous Galerkin  method for the numerical approximation of this coupled system, providing a high-order, stable discretization that efficiently handles the multiphysics nature of the problem. We establish stability analysis and derive optimal $hp$-error estimates for the semi-discrete formulation. The theoretical convergence rates are validated through comprehensive numerical experiments, demonstrating the method's accuracy and robustness across various test cases, including wave propagation in heterogeneous media with mixed viscoelastic properties. 
\end{abstract}
\bigskip

\bigskip



\section{Introduction}\label{s:introduction}

Porous and thermally responsive viscoelastic media are commonly found in geomechanics, biomechanics, and various engineered systems.  Since Biot \cite{Biot1941} and Terzaghi \cite{terzaghi1943} first quantified the interplay between fluid pressure and solid deformation, poroelasticity has become the cornerstone for modelling soil consolidation, reservoir compaction \cite{russell1983finite, gai2004}, and geothermal  operations \cite{ghassemi2024role}. Incorporating viscoelastic rheology is essential to reproduce secondary consolidation in clays \cite{loula1996} and stress-relaxation in biological tissues such as cartilage and brain matter \cite{mow1980, bociu2023, quarteroni2023}. In earthquake modelling and exploration geophysics, thermo-poro-viscoelastic effects govern wave attenuation and slow-mode diffusion \cite{boukamel2001thermo, cavallini2019, morency}.  These wide-ranging use cases motivate robust numerical tools capable of accurately capturing coupled poro/thermo-viscoelastic dynamics.

Numerical discretizations for coupled poro- and thermoelastic systems have undergone significant advancement over the past two decades. For dynamic problems, a diverse spectrum of computational approaches has been developed and refined. These include classical methods such as finite difference \cite{gaspar2003}, finite element \cite{santos1986II}, boundary element \cite{chen1995}, finite volume \cite{Lemoine2013}, and spectral techniques \cite{morency}.

Recent advancements have substantially expanded these computational methods. Among the most significant developments are stabilized mixed methods \cite{Lee2023} and high-order space-time Galerkin formulations \cite{bause2024}. Advanced polytopal discontinuous Galerkin schemes with adaptive refinement capabilities provide effective solutions for problems involving complex geometries \cite{antonietti2023discontinuous, antoniettiIMA, bonetti2024, bonetti2023numerical}. Similarly, hybridized discontinuous Galerkin (HDG) methods \cite{fu2019, meddahi2025} support adaptive refinement while achieving optimal convergence with a reduced number of coupled degrees of freedom. 

For thermo-viscoelasticity, Bonetti \& Corti \cite{bonetti2024} developed a weighted symmetric interior-penalty DG scheme on general polytopal meshes that handles strong heterogeneities robustly. Their displacement/temperature formulation employs the Kelvin-Voigt viscoelasticity model, which combines elastic and viscous responses. However, this model cannot accurately reproduce finite-time stress-relaxation or creep-recovery phenomena essential in many applications. The more general Zener model proves more suitable for practical situations. However, displacement-based formulations of this model lead to hereditary convolution terms that render the problem non-local in time. To overcome this limitation, we adopt the stress-based formulation presented in \cite{Becache, meddahi2023hp} for the linear viscoelastic Zener's model, which avoids the non-local terms while maintaining the model's advantages.

The present work develops a high-order hybridizable discontinuous Galerkin method for a multi-field formulation that couples a Zener (standard-linear-solid) rheology with either Biot or Maxwell-Cattaneo transport laws. Following \cite{meddahi2023hp}, we adopt a stress-based formulation and rewrite the problem as a first-order system in time. The primary unknowns in this system are the solid velocity, the elastic and viscous stress tensors, along with either pressure and filtration velocity (in the Biot case) or temperature and heat flux vector field (in the Maxwell-Cattaneo case). The proposed HDG scheme is designed for simplicial meshes with hanging nodes, employing hybrid variables for the solid velocity and the gradient of the transport field—either filtration velocity or heat flux. This choice reduces the size of the globally coupled unknowns to the skeletal degrees of freedom and enables static condensation without parameter tuning. The resulting HDG method exhibits uniform energy stability, achieves optimal $hp$-error estimates, and naturally extends to heterogeneous media containing both purely elastic and viscoelastic subdomains—capabilities that, as far as the authors are aware, have not been reported for Zener-type poro/thermo-viscoelastic problems.

The remainder of this article is organized as follows. Section \ref{sec:model} describes the fully dynamic poro-/thermo-viscoelastic system. Section~\ref{sec:variational} formulates the model problem as a first-order time evolution problem in a suitable functional setting and employs the theory of strongly continuous semigroups to prove well-posedness.  In Section \ref{sec:semi-discrete} we introduce the semidiscrete HDG scheme, and in Section \ref{sec:convergence} we establish its stability and derive optimal $hp$-error estimates. Finally, Section \ref{sec:numresults} presents a suite of numerical experiments—ranging from convergence tests to geophysical wave-propagation benchmarks—that confirm our theoretical rates and illustrate the method's performance in heterogeneous viscoelastic media. Appendix~\ref{appendix} gathers the essential $hp$-approximation results on simplicial meshes that underlie our error analysis.

Throughout the rest of this paper, we shall use the letter $C$  to denote generic positive constants independent of the mesh size $h$ and the polynomial degree $k$. These constants may represent different values at different occurrences. Moreover, given any positive expressions $X$ and $Y$ depending on $h$ and $k$, the notation $X \,\lesssim\, Y$  means that $X \,\le\, C\, Y$.

\section{The dynamic and linear poro/thermo-viscoelastic model problem}\label{sec:model} 

Let $\Omega \subset \bbR^d$ be a bounded polygonal or polyhedral domain, where $d \in\set{2,3}$ denotes the spatial dimension. Consider the following coupled system of partial differential equations on the time interval \((0,T]\), where \(T>0\) denotes the final time:
\begin{subequations}\label{eq:poro_thermo_visco}
   \begin{align}
     \rho \dot{\bu} - \bdiv \bigl(\bsig - \alpha \psi I_d \bigr)  &=  \bF
     \quad  \text{in } \Omega \times (0,T], \label{eq:ptv-a}
     \\
     \dot{\bsig}_E  &= \cC \beps(\bu)
     \quad \text{in } \Omega \times (0,T], \label{eq:ptv-b}
     \\
     \omega\,\dot{\bsig}_V  +   \bsig_V  &= (\cD  - \cC)\,\beps(\bu)
     \quad \text{in } \Omega \times (0,T], \label{eq:ptv-c}
     \\
     \chi\,\dot{\bp}  +  \beta\,\bp  +  \bnabla \psi  &= \mathbf{0}
     \quad \text{in } \Omega \times (0,T], \label{eq:ptv-d}
     \\
     s \dot{\psi}  +  \div\bp  +  \alpha \div\bu &= g
     \quad \text{in } \Omega \times (0,T]. \label{eq:ptv-e}
   \end{align}
\end{subequations}
Here, $I_d \in \mathbb{R}^{d \times d}$ is the identity in the space of real matrices $\mathbb{R}^{d\times d}$ and $\beps(\bu):= \frac{1}{2}\left[\bnabla\bu+(\bnabla\bu)^{\mt}\right] : \Omega \to \mathbb{R}^{d\times d}_{\text{sym}}$ is the linearized strain tensor, where $\mathbb{R}^{d \times d}_{\text{sym}} \coloneq \{\btau \in \mathbb{R}^{d \times d} \mid \btau = \btau^{\mt}\}$.

The coupled system \eqref{eq:poro_thermo_visco} presents a unified mathematical framework that captures both poro-viscoelastic and thermo-viscoelastic phenomena through a stress-velocity formulation, following the approach of \cite{Becache, meddahi2023hp}. This stress-velocity formulation  for the solid eliminates time-domain convolution terms, thereby offering advantages in both theoretical analysis and numerical implementation, see \cite{meddahi2023hp} for more details. 

Equation \eqref{eq:ptv-a} represents the momentum balance for the solid phase. In this equation, $\bu$ denotes the velocity field, and the effective stress tensor $\bsig$ follows the Zener viscoelastic model, decomposed as $\bsig = \bsig_E + \omega \, \bsig_V$, where $\omega > 0$ is the relaxation time. The term $\bF:\Omega\times[0, T] \to \bbR^d$ accounts for external body forces. The density parameter $\rho$ has different interpretations: in thermoelasticity, it represents the solid density $\rho_S$, while in poroelasticity, it is the composite density $\rho = \phi \rho_F + (1 - \phi) \rho_S$, where $\phi$ is the porosity, and $\rho_F$ and $\rho_S$ are the fluid and solid densities, respectively.
   
The constitutive laws governing the evolution of the elastic and viscous stress components are given by equations \eqref{eq:ptv-b} and \eqref{eq:ptv-c}. These involve the symmetric and positive definite fourth-order tensors $\cC$ and $\cD$, with the condition that $\cD - \cC$ is also positive definite to ensure energy dissipation. The coupling between the solid mechanics and the secondary field (fluid or thermal) is established through the term $\alpha \, \psi \, I_d$ in the momentum equation, where $\psi$ represents either the fluid pressure or the temperature.
   
The physical interpretation of equation \eqref{eq:ptv-d} varies with the context. In poroelasticity, it corresponds to Darcy's law, describing the filtration velocity $\bp$ of the interstitial fluid, with $\beta = \eta/\kappa$ representing the ratio of fluid viscosity to permeability, and $\chi = \rho_F/\phi$ relating to the effective fluid mass. The coefficient $\alpha \in (0, 1]$ is the Biot-Willis coefficient, quantifying the coupling between the solid deformation and fluid pressure. Alternatively, in thermoelasticity, equation \eqref{eq:ptv-d} represents the Maxwell-Vernotte-Cattaneo law for heat conduction, where $\chi = \gamma/\theta$ and $\beta = 1/\theta$, with $\gamma>0$ being the thermal relaxation time, $\theta$ the effective thermal conductivity, and $\alpha > 0$ the thermal stress coefficient.
   
Finally, equation \eqref{eq:ptv-e} expresses a linearized conservation principle: mass conservation for the fluid in poroelasticity or energy conservation in thermoelasticity. The parameter $s$ represents either the specific storage coefficient or the thermal dilatation effect, while $g$ denotes the corresponding source term (fluid injection/extraction or heat generation).

To complete the problem formulation, we specify the initial conditions at time \(t=0\):
\begin{equation}\label{eq:IC}
   \bu(0) = \bu^0,\quad \bsig_E(0) = \bsig_E^0, \quad \bsig_V(0) = \bsig_V^0, \quad \bp(0) = \bp^0, \quad \text{and}\quad \psi(0) = \psi^0 \quad \text{in $\Omega$},
\end{equation}
where \(\bu^0,\bsig_E^0,\bsig_V^0,\bp^0,\psi^0\) are given initial data defined on the domain \(\Omega\).

Finally, we impose mixed Dirichlet and Neumann boundary conditions. Let \( \partial\Omega = \Gamma^u_D\cup\Gamma^u_N\) and \(\partial\Omega = \Gamma^\psi_D\cup\Gamma^\psi_N\) be two partitions of the boundary \(\partial\Omega\) into disjoint Dirichlet and Neumann parts, respectively, with the condition that the Dirichlet boundary sets \(\Gamma^u_D\) and \(\Gamma^\psi_D\) have positive measure. Denoting by \(\bn\) the outward unit normal vector on \(\partial \Omega\), the boundary conditions are given by:
\begin{equation}\label{eq:BC}
   \begin{aligned}
      \bu &= \mathbf 0 \quad \text{on $\Gamma^u_D\times (0, T]$},
      &\quad (\bsig - \alpha \psi I_d)\bn &= \mathbf 0 \quad \text{on $\Gamma^u_N\times (0, T]$},
      \\
        \psi &= 0 \quad \text{on $\Gamma^\psi_D\times (0, T]$}, &\quad
      \bp\cdot \bn &=  0 \quad \text{on $\Gamma^\psi_N\times (0, T]$}.
   \end{aligned}
\end{equation}

\section{The continuous problem}\label{sec:variational}

\subsection{Notation and functional spaces}
For any $m, n \in \mathbb{N}$ and $s\in \bbR$, $H^s(D, \mathbb{R}^{m\times n})$ represents the usual Hilbertian Sobolev space of functions with domain $D$ and values in $\mathbb{R}^{m\times n}$. In the case $m=n=1$, we simply write $H^s(D)$. The norm of $H^s(D, \mathbb{R}^{m\times n})$ is denoted by $\norm{\cdot}_{s,D}$ and the corresponding semi-norm is written as $|\cdot|_{s,D}$. We adopt the convention  $H^0(D, \mathbb{R}^{m\times n}):=L^2(D,\mathbb{R}^{m\times n})$ and let $(\cdot, \cdot)_D$ be the inner product in $L^2(D, \mathbb{R}^{m\times n})$, i.e.,
\begin{equation*}
	  \inner{\bsig, \btau}_D \coloneq  \int_D \bsig:\btau,\quad \forall \bsig, \btau\in L^2(D,\mathbb{R}^{m\times n}),
\end{equation*} 
where the component-wise inner product of two matrices $\bsig = (\sigma_{ij})$ and $\btau = (\tau_{ij}) \in \mathbb{R}^{m \times n}$ is given by $\bsig : \btau \coloneq \sum_{i,j} \sigma_{ij} \tau_{ij}$.

To establish a suitable framework for formulating the first-order system in time \eqref{eq:poro_thermo_visco}, we define the pivot Hilbert spaces $\mathcal{H}_1\coloneq [L^2(\Omega, \mathbb R^{d})]^2$ and $\mathcal{H}_2 \coloneq  [L^2(\Omega,\mathbb{R}^{d\times d}_{\text{sym}})]^2 \times   L^2(\Omega)$. We equip $\mathcal{H}_1$ with the inner product $\inner{ (\bu, \bp), (\bv, \bq) }_{\mathcal{H}_1} \coloneq  \inner{\rho\bu, \bv}_{\Omega}+ \inner{\chi\bp, \bq}_{\Omega}$. Consequently, the corresponding norm is defined as $\norm{(\bv, \bq)}^2_{\mathcal{H}_1} \coloneq  \inner{ (\bv, \bq), (\bv, \bq) }_{\mathcal{H}_1}$ and it satisfies 
\begin{equation}\label{bound:rho}
   \rho^-  (\norm{\bv}^2_{0,\Omega} + \norm{\bq}_{0,\Omega}^2) \leq 	\norm{(\bv, \bq)}^2_{\mathcal{H}_1} \leq \rho^+  (\norm{\bv}^2_{0,\Omega} + \norm{\bq}_{0,\Omega}^2) \quad \forall (\bv, \bq) \in \mathcal{H}_1,
\end{equation}  
with $\rho^- = \min\{\rho, \chi\}$ and $\rho^+ = \max\{\rho, \chi\}$.
Additionally, we define $\mathcal{A} \coloneq  \mathcal{C}^{-1}$ and $\cG \coloneq  (\cD - \cC)^{-1}$, and equip $\mathcal{H}_2$ with the inner product given by 
\[
\inner{ (\bsig_E, \bsig_V, \psi), (\btau_E, \btau_V, \varphi) }_{\mathcal{H}_2} \coloneq   \inner{\mathcal{A}\bsig_E, \btau_E}_\Omega +\inner{\omega\mathcal{G}\bsig_V, \omega\btau_V}_\Omega + \inner{s \psi, \varphi }_\Omega,
\]  
for all $(\bsig_E, \bsig_V, \psi), (\btau_E, \btau_V, \varphi) \in \mathcal{H}_2$.
Due to our assumptions on $\cC$ and $\cD$, there exist constants $a^+ > a^- >0$ such that the norm $\norm{(\btau_E, \btau_V, \varphi) }^2_{\mathcal{H}_2}  \coloneq   \inner{\mathcal{A}\btau_E, \btau_E}_\Omega + \inner{\omega\mathcal{G}\btau_V, \omega\btau_V}_\Omega  + \inner{s \varphi,\varphi }_\Omega$ satisfies 
\begin{equation}\label{normH}
	a^- \left( \norm{\btau_E}^2_{0, \Omega} + \norm{\btau_V}^2_{0, \Omega}  + \norm{\varphi}^2_{0, \Omega}  \right) \leq \norm{(\btau_E, \btau_V, \varphi)  }_{\cH_2}^2 \leq a^+ \left( \norm{\btau_E}^2_{0, \Omega} + \norm{\btau_V}^2_{0, \Omega}  + \norm{\varphi}^2_{0, \Omega}  \right),
\end{equation} 
for all $(\btau_E, \btau_V, \varphi)    \in \mathcal{H}_2$.

To enforce the essential boundary conditions \eqref{eq:BC} within the energy space for the solid velocity, we introduce the closed subspace of $H(\div, \Omega) \coloneq \set{\bv \in L^2(\Omega, \bbR^d);\ \div \bv \in L^2(\Omega)}$ given by 
\[
H_N(\div, \Omega) := \set{ \bv\in H(\div, \Omega); \quad \dual{\bv\cdot \bn,q}_{\Gamma}= 0 	\quad \forall q\in H^1_D(\Omega) }.
\]
It consists of vector fields in $H(\div, \Omega)$ with a free normal component on $\Gamma^\psi_N$. Here,  $H^1_D(\Omega):=\set{\varphi\in H^1(\Omega);\ \varphi|_{\Gamma^\psi_D} = 0}$, and $\dual{\cdot, \cdot}_\Gamma$ represents the duality pairing between $H^{\sfrac12}(\Gamma)$ and $H^{-\sfrac12}(\Gamma)$. Similarly, we let $H^1_D(\Omega,\bbR^d):=\set{\bv\in H^1(\Omega,\bbR^d);\ \bv|_{\Gamma^u_D} = \mathbf 0}$ and define
\[
H_N(\bdiv, \Omega, \mathbb{R}^{d\times d}_{\text{sym}}) := \set{ \btau\in H(\bdiv, \Omega, \mathbb{R}^{d\times d}_{\text{sym}}); \quad 
	\left\langle\btau\bn,\bv\right\rangle_{\Gamma}= 0 
	\quad \forall\bv\in H^1_D(\Omega,\bbR^d) },
\]
as the closed subspace of $H(\bdiv, \Omega, \mathbb{R}^{d\times d}_{\text{sym}}) \coloneq \set{\btau \in L^2(\Omega, \mathbb{R}^{d\times d}_{\text{sym}});\ \bdiv \btau \in L^2(\Omega, \bbR^d)}$  satisfying a stress--free boundary condition on $\Gamma^u_N$.

\subsection{Abstract formulation}
Collecting the unknowns in $\mathbf y(t):= (\bsig_E, \bsig_V, \psi, \bu, \bp)^{\mt}$, system \eqref{eq:poro_thermo_visco}-\eqref{eq:IC} rewrites compactly as
\begin{align}\label{eq:IVP}
    \begin{split}
        \dot{\mathbf{y}}(t) &= A \mathbf{y}(t) + \mathbf F(t) \quad \text{in $\cH_1 \times \cH_2$},
   \\
   \mathbf y(0) &= \mathbf y^0
    \end{split}
\end{align}
where $\mathbf F(t) \coloneq (\tfrac{1}{\rho}\bF, \mathbf 0, \mathbf 0, \mathbf 0,  \tfrac{1}{s}g)$ and $\mathbf y^0 \coloneq  (\bu^0, \bp^0, \bsig_E^0, \bsig_V^0, \psi^0)^{\mt}$. 

The infinitesimal generator $A$ of the strongly continuous semigroup on $\cH_1 \times \cH_2$ associated with the evolution system \eqref{eq:IVP} is given by
\begin{align*}
A\left( \bu,\bp , \bsig_E, \bsig_V, \psi  \right) &\coloneq 
\Big( A_1(\bsig_E, \bsig_V, \psi) + (0,\tfrac{\beta}{\chi}\bp),\ A_2 (\bu,\bp) + (0, \tfrac{1}{\omega} \bsig_V , 0) \Big) 
\\
A_1(\bsig_E, \bsig_V, \psi) &\coloneq  \Big( - \tfrac{1}{\rho}\bdiv(\bsig_E + \omega\bsig_V - \alpha \psi \mathrm{I}_d),\ \tfrac{1}{\chi} \nabla \psi  \Big)
    \\
    A_2 (\bu,\bp) &\coloneq \Big(- \mathcal{C} \beps(\bu),\ -\tfrac{1}{\omega}(\cD - \cC)\beps(\bu),\ \tfrac{1}{s} ( \div \bp +\alpha \div \bu )  \Big).
\end{align*}

The domain of the operator $A$ is the subspace  $\mathcal{X}_1 \times \mathcal{X}_2 \subset \mathcal{H}_1 \times \mathcal{H}_2$, where  
\begin{align*}
	\mathcal{X}_1 &\coloneq H_D^1(\Omega,\mathbb{R}^d) \times H_N(\div,\Omega)
	\\
	\mathcal{X}_2 &\coloneq \set*{ (\btau_E, \btau_V, \varphi) \in [L^2(\Omega,\mathbb{R}^{d\times d}_{\text{sym}})]^2\times  H_D^1(\Omega);\quad     \btau_E +\omega \btau_V - \alpha \varphi I_d  \in H_N(\bdiv, \Omega, \mathbb{R}^{d\times d}_{\text{sym}})  }.
\end{align*}
We endow $\mathcal{X}_1$ and $\mathcal{X}_2$ with the inner products
\begin{align*}
&\inner*{(\bu,\bp) , (\bv,\bq)  }_{\mathcal{X}_1} \coloneq   \inner*{(\bu,\bp) , (\bv,\bq)  }_{\mathcal{H}_1} + \inner*{\beps(\bu), \beps(\bv)}_\Omega + \inner*{\tfrac{1}{s}(\div\bp + \alpha\div \bu), \div\bq + \alpha\div \bv}_\Omega,
\\
   &\inner*{ (\bsig_E, \bsig_V, \psi), (\btau_E, \btau_V, \varphi) }_{\mathcal{X}_2} 
   \coloneq \inner*{ (\bsig_E, \bsig_V, \psi), (\btau_E, \btau_V, \varphi) }_{\mathcal{H}_2} 
   \\
    &\qquad + \inner*{ \tfrac{1}{\rho} \bdiv(\bsig_E +\omega \bsig_V - \alpha \psi I_d ),    \bdiv(\btau_E +\omega \btau_V - \alpha \varphi I_d)  }_\Omega
   + \inner*{ \tfrac{1}{\chi} \nabla \psi ,  \nabla \varphi  }_\Omega,
\end{align*}
respectively. The completeness of the spaces $\mathcal{X}_1$ and $\mathcal{X}_2$ when equipped with the norms $\norm{\cdot}_{\mathcal{X}_1}$ and $\norm{\cdot}_{\mathcal{X}_2}$, respectively, can be established following the approach of \cite[Section 3.2]{meddahi2025}.

\subsection{Well-posedness analysis}
For a separable Banach space $V$, we denote by $\mathcal{C}^0_{[0,T]}(V)$ the Banach space of continuous functions $f: [0,T] \to V$. Moreover, $\mathcal{C}^1_{[0,T]}(V)$ represents the subspace of $\mathcal{C}^0_{[0,T]}(V)$ consisting of functions $f$ with strong derivatives $\frac{d f}{dt}$ in $\mathcal{C}^0_{[0,T]}(V)$.

The Green formulas
\begin{equation}\label{green}
	\inner{\bv, \nabla \varphi}_\Omega + \inner{\div \bv, \varphi}_\Omega = 0
	\quad \text{and} \quad 
	\inner{\btau, \beps(\bw)}_\Omega + \inner{\bdiv \btau, \bw}_\Omega = 0,
\end{equation}
hold for arbitrary pairs $(\bv, \varphi) \in H_N(\div, \Omega)\times H^1_D(\Omega) $ and $(\btau, \bw) \in H_N(\bdiv, \Omega, \mathbb{R}^{d\times d}_{\text{sym}})\times H^1_D(\Omega, \bbR^d)$. They permit us to obtain the identity
\begin{equation}\label{eq:monotone0}
	\inner*{ A\left( (\bu,\bp) , (\bsig_E, \bsig_V, \psi)  \right),  \left( (\bu,\bp) , (\bsig_E, \bsig_V, \psi)  \right) }_{\mathcal{H}_1 \times \mathcal{H}_2} 
   = \inner{\cG \bsig_V, \omega\bsig_V}_\Omega + \inner{\beta \bp, \bp}_\Omega \geq 0,
\end{equation}
which holds for all $(\bu, \bp) \in \mathcal{X}_1$ and $(\bsig_E, \bsig_V, \psi) \in \mathcal{X}_2$. This identity proves that the operator $A: \mathcal{X}_1 \times \mathcal{X}_2 \to \mathcal{H}_1 \times \mathcal{H}_2$ is monotone. To establish the well-posedness of the initial boundary value problem \eqref{eq:IVP} via the Hille--Yosida theorem, we must  prove the surjectivity of $I_{\mathcal{X}_1 \times \mathcal{X}_2} + A: \mathcal{X}_1 \times \mathcal{X}_2 \to  \mathcal{H}_1 \times \mathcal{H}_2$, where $I_{\mathcal{X}_1 \times \mathcal{X}_2}$ denotes the identity operator in $\mathcal{X}_1 \times \mathcal{X}_2$. The proof follows arguments similar to those in \cite[Lemma 1]{meddahi2025}.

\begin{theorem}\label{thm:Hille-Yosida}
For all $\mathbf{F}  \in \mathcal{C}^1_{[0,T]}(\cH_1 \times \cH_2)$ and $\mathbf y^0 \in \mathcal{X}_1\times \mathcal{X}_2$, there exists a unique solution $\mathbf y \in \mathcal{C}^1_{[0,T]}(\cH_1 \times \cH_2) \cap \mathcal{C}^0_{[0,T]}(\mathcal{X}_1\times \mathcal{X}_2)$ to the initial boundary value problem \eqref{eq:IVP}. Moreover, there exists a constant $C>0$ such that  
\begin{equation}\label{eq:stab}
      \max_{[0,T]}\norm{\mathbf y(t)}_{\mathcal{H}_1 \times \cH_2} \leq C \max_{[0,T]}\norm{\mathbf F(t)}_{0,\Omega} +   \norm{\mathbf{y}^0}_{\mathcal{H}_1\times \cH_2}.
\end{equation}
\end{theorem}
\begin{proof}
See \cite[Theorem 76.7]{ErnBook2021III} for a detailed proof.
\end{proof}

We define the bounded bilinear form $B: \cX_1 \times \cX_2 \to \bbR$ as
\[
B( (\bv, \bq), (\btau_E, \btau_V, \varphi)) \coloneq   \inner{\bdiv(\btau_E + \omega\btau_V - \alpha \varphi I_d), \bv}_\Omega  
-  \inner{ \nabla \varphi ,\bq}_\Omega.
\]
In the following sections, we will develop a finite element discretization method based on the following weak form of \eqref{eq:IVP}: Find $(\bu, \bp) \in \mathcal{C}_{[0,T]}^1(\cH_1) \cap \mathcal{C}^0_{[0,T]}(\cX_1)$ and  $(\bsig_E, \bsig_V, \psi) \in \mathcal{C}^1_{[0,T]}(\cH_2) \cap \mathcal{C}^0_{[0,T]}(\cX_2)$ satisfying 
\begin{align}\label{eq:weakform}
	\begin{split}
		\inner*{(\dot\bu, \dot\bp) , (\bv, \bq)}_{\cH_1} + \inner*{\beta\bp,\bq}_\Omega  -B( (\bv, \bq), (\bsig_E, \bsig_V, \psi))
		&= \inner*{\bF, \bv}_\Omega 
		\\
	\inner*{(\dot\bsig_E, \dot\bsig_V, \dot\psi), (\btau_E, \btau_V, \varphi) }_{\cH_2} + \inner*{\cG \bsig_V, \omega\btau_V}_\Omega
	+  B( (\bu, \bp), (\btau_E, \btau_V, \varphi))
	&= \inner*{g, \varphi}_\Omega,
	\end{split}
\end{align}
for all $(\bv, \bq)\in \cX_1$ and all $(\btau_E, \btau_V, \varphi)\in \cX_2$, and subject to the initial conditions given in \eqref{eq:IVP}. 

\section{The semi-discrete problem}\label{sec:semi-discrete}

For the sake of simplicity, from now on we assume that $\Gamma^u_D = \Gamma^\psi_D = \Gamma$. As a result, the boundary conditions \eqref{eq:BC} become 
\begin{equation}\label{newBC}
		\bu = \mathbf 0  \quad \text{and} 
	\quad 
	\psi = 0 \quad \text{on $ \Gamma\times (0, T]$}.
\end{equation}   

Let $\cT_h$ denote a shape-regular mesh of the domain $\bar \Omega$, consisting of tetrahedra and/or parallelepipeds in three dimensions ($d=3$), or triangles and/or quadrilaterals in two dimensions ($d=2$). The mesh $\cT_h$ may contain hanging nodes. For each element $K \in \cT_h$, we denote its diameter by $h_K$, and define the mesh size parameter $h := \max_{K\in \cT_h} \{h_K\}$.

We say that a closed subset $F\subset \overline{\Omega}$ is an interior edge/face if it has positive $(d-1)$-dimensional measure and arises as the intersection of two distinct elements, that is, $F =\bar K\cap \bar K'$ for some $K, K' \in \cT_h$. Similarly, a closed subset $F\subset \overline{\Omega}$ is termed a boundary edge/face if it constitutes both an edge/face of some element $K \in \cT_h$ and part of the domain boundary, namely $F =  \bar K\cap \partial \Omega$. We denote by $\cF_h^0$ and $\cF_h^\partial$ the collections of interior and boundary edges/faces respectively, and set $\cF_h = \cF_h^0\cup \cF_h^\partial$. 

For any edge/face $F\in\cF_h$, we denote its diameter by $h_F$. The mesh $\cT_h$ is assumed to be locally quasi-uniform with constant $\gamma>0$, meaning that for all mesh sizes $h$ and all elements $K\in \cT_h$:
\begin{equation}\label{reguT}
	h_F \leq h_K\leq \gamma h_F\quad \forall F\in \cF(K),
\end{equation}
where $\cF(K)$ denotes the set of edges/faces of element $K$. This local quasi-uniformity is invoked for the projection error estimate of Lemma~\ref{maintool2}.

\subsection{Broken Sobolev spaces}

For any non-negative real number $s\geq 0$, we define the broken Sobolev space with respect to the partition $\cT_h$ of $\bar \Omega$ as
\[
 H^s(\cT_h,E):=
 \set{\bv \in L^2(\Omega, E): \quad \bv|_K\in H^s(K, E)\quad \forall K\in \cT_h},\quad \text{for $E \in \set{ \bbR, \bbR^d, \bbS}$}. 
\]
For consistency with earlier notation, we write $H^0(\cT_h,E) = L^2(\cT_h,E)$ and simplify $H^s(\cT_h,\bbR)$ to $H^s(\cT_h)$. We introduce on $L^2(\cT_h, E)$, with $E\in \set{\bbR, \bbR^d, \bbS}$, the inner product
\[
 \inner{\psi, \varphi}_{\cT_h} := \sum_{K\in \cT_h} \inner{\psi, \varphi}_{ K}\quad \forall  \psi, \varphi \in L^2(\cT_h, E),
\]
and denote its induced norm by $\norm{\psi}^2_{0,\cT_h}:= \inner{\psi, \psi}_{\cT_h}$. 

Next, we consider the set $\partial \cT_h :=\set{\partial K;\ K\in \cT_h}$ containing all element boundaries. On the space $L^2(\partial \cT_h,\bbR^d)$ of vector-valued functions that are square-integrable on each $\partial K\in \partial \cT_h$, we define the inner product and its associated norm:
\[
  \dual{\bu, \bv}_{\partial \cT_h} := \sum_{K\in \cT_h} \dual{\bu, \bv}_{\partial K}, 
  \quad \text{and} \quad
  \norm{\bv}^2_{0, \partial \cT_h}:= \dual{\bv, \bv}_{\partial \cT_h}
  \quad
  \forall  \bu, \bv\in L^2(\partial \cT_h,\bbR^d),
\]
where $\dual{\bu, \bv}_{\partial K} := \sum_{F\in \cF(K)} \int_F\bu\cdot \bv$. 

Finally, we consider the space $L^2(\cF_h,\bbR^d):= \prod_{F\in \mathcal{F}_h} L^2(F, \bbR^d)$, equipped with the inner product
\[
(\bu, \bv)_{\cF_h} := \sum_{F\in \cF_h} \int_F\bu\cdot \bv \quad \forall \bu,\bv\in L^2(\cF_h,\bbR^d),
\]
and its corresponding norm $\norm*{\bv}^2_{0,\cF_h}:= (\bv,\bv)_{\cF_h}$. 

It is important to keep in mind that functions in $L^2(\partial \cT_h,\bbR^d)$ carry two distinct values on each interior face $F$ (one from each adjacent element). In contrast, functions in $L^2(\cF_h,\bbR^d)$ take a unique value on each face $F$.

To formulate the HDG method, we need to introduce product spaces that accommodate both volume and trace variables. For $r>\sfrac12$, we define:
\begin{align*}
	\mathcal{Q} &\coloneq \left\{ \overrightarrow{\bq} = (\bq, \hat\bq) \in [H(\div,\cT_h)\cap H^r(\cT_h,\mathbb{R}^d)] \times L^2(\cF_h,\mathbb{R}^{d}) \right\},
\\
\mathcal{U} &\coloneq \left\{\overrightarrow{\bv} = (\bv, \hat\bv) \in H^1(\cT_h,\mathbb{R}^d) \times L^2(\cF^0_h,\mathbb{R}^{d}) \right\},
\end{align*}
where $L^2(\cF^0_h,\mathbb{R}^{d})$ denotes the space of vector-valued functions in $L^2(\cF_h,\mathbb{R}^{d})$ that vanish on boundary faces:
\[
L^2(\cF^0_h,\mathbb{R}^{d}) \coloneq \left\{\boldsymbol{\phi}\in L^2(\cF_h,\mathbb{R}^{d}) \mid \boldsymbol{\phi}|_F = \mathbf{0},\ \forall F\in \cF_h^\partial \right\}.
\]

We define the jumps $\jump{\overrightarrow{\bq}} = \bq - \hat{\bq}$ and $\jump{\overrightarrow{\bv}} = \bv - \hat{\bv}$, and equip the spaces $\mathcal{Q}$  and $\mathcal{U}$ with the semi-norms:
\begin{equation}\label{norm_spaces}
\abs*{\overrightarrow{\bq}}^2_{\mathcal{Q}} = \norm{\div \bq}^2_{0,\cT_h} + \norm{\tfrac{k+1}{h_\cF^{\sfrac{1}{2}}}\jump{\overrightarrow{\bq}}}^2_{0, \partial \cT_h}
\quad \text{and} \quad 
\abs*{\overrightarrow{\bv}}^2_{\mathcal{U}} = \norm{\beps(\bv)}^2_{0,\cT_h} + \norm{\tfrac{k+1}{h_\cF^{\sfrac{1}{2}}}\jump{\overrightarrow{\bv}}}^2_{0, \partial \cT_h},
\end{equation}
where $h_\cF\in \cP_0(\cF_h)$ is the piecewise constant function defined by $h_\cF|_F = h_F$ for all $F \in \cF_h$.

\subsection{The HDG method}
We denote by $\cP_m(D)$ the space of polynomials of degree at most $m\geq 0$ on $D$ if $D$ is a triangle/tetrahedron, and the space of polynomials of degree at most $m$ in each variable  if $D$ is a quadrilateral/parallelepiped. For vector and tensor-valued polynomials, we write $\cP_m(D, E)$ to denote the space of $E$-valued functions whose components belong to $\cP_m(D)$, where $E$ is either $\bbR^d$ or $\bbS$. On the mesh $\cT_h$, we define the space of piecewise-polynomial functions as
\[
 \cP_m(\cT_h) := 
 \set{ v\in L^2(\cT_h): \ v|_K \in \cP_m(K),\ \forall K\in \cT_h }.
 \]
Similarly, on the skeleton $\cF_h$, we introduce
 \[
 \cP_m(\cF_h) := 
 \set{ \hat\bv\in L^2(\cF_h): \ \hat\bv|_F \in \cP_m(F),\ \forall F\in \cF_h }.
 \]

For vector and tensor-valued piecewise polynomials, we use the notation $\cP_m(\cT_h, E)$ to denote the subspace of $L^2(\cT_h, E)$ whose components belong to $\cP_m(\cT_h)$, where $E\in \set{\bbR^d, \bbS}$. Analogously, $\cP_m(\cF_h, \bbR^d)$ represents the subspace of $L^2(\cF_h, \bbR^d)$ with components in $\cP_m(\cF_h)$.

On element boundaries, we define
\[
  \cP_m(\partial \cT_h, \bbR^d) := \set*{\bv\in L^2(\partial \cT_h, \bbR^d);\ \bv|_{\partial K}\in  \cP_m(\partial K, \bbR^d),\ \forall K\in \cT_h},
\]
where $\cP_m(\partial K, \bbR^d)$ is the product space $\prod_{F\in \cF(K)} \cP_m(F, \bbR^d)$.

For $k\geq 0$, we introduce the finite-dimensional subspaces of $\mathcal{H}_1$ and $\mathcal{H}_2$ given by
\[
	\mathcal{H}_{1,h} \coloneq  \cP_{k+1}(\cT_h,\bbR^{d})\times \cP_{k+1}(\cT_h,\bbR^{d})\quad \text{and} \quad \mathcal{H}_{2,h} \coloneq \cP_{k}(\cT_h,\mathbb{R}^{d\times d}_{\text{sym}}) \times \cP_{k}(\cT_h,\mathbb{R}^{d\times d}_{\text{sym}}) \times \cP_{k}(\cT_h),
\]
respectively. We also define the discrete counterparts of spaces $\mathcal{Q}$ and $\mathcal{U}$ as:
\[
\mathcal{Q}_h \coloneq \cP_{k+1}(\cT_h,\bbR^{d}) \times \cP_{k+1}(\cF_h, \mathbb{R}^d) 
\quad
\text{and}
\quad 
\mathcal{U}_h \coloneq \cP_{k+1}(\cT_h,\bbR^{d}) \times \cP_{k+1}(\cF^0_h, \mathbb{R}^d),
\]
where $\cP_{k+1}(\cF^0_h, \mathbb{R}^d)$ consists of polynomial functions vanishing on boundary faces:
\[
\cP_{k+1}(\cF^0_h, \mathbb{R}^d) \coloneq  \{\boldsymbol{\phi} \in \cP_{k+1}(\cF_h, \mathbb{R}^d) \mid \boldsymbol{\phi}|_F = \mathbf{0},\ \forall F\in \cF_h^\partial\}.
\]

We propose the following HDG space-discretization method for problem \eqref{eq:weakform}:  find    $\overrightarrow{\bu}_h = (\bu_{h}, \hat{\bu}_{h} ) \in \cC^1_{[0,T]}(\mathcal{U}_h)$, $\overrightarrow{\bp}_h = (\bp_{h}, \hat{\bp}_{h} ) \in \cC^1_{[0,T]}(\mathcal{Q}_h)$ and $(\bsig_{E,h}, \bsig_{V,h}, \psi_h) \in  \cC^1_{[0,T]}(\cH_{2,h})$    satisfying
\begin{align}\label{sd}
	\begin{split}
		\inner*{(\dot{\bu}_{h}, \dot{\bp}_{h} ), (\bv,\bq)}_{\cH_1} &+ 
		\inner{(\dot{\bsig}_{E,h}, \dot{\bsig}_{V,h}, \dot \psi_h), (\btau_E, \btau_V,\varphi) }_{\cH_2} + \inner{\cG \bsig_{V,h}, \omega\btau_V}_\Omega  + \inner{\beta \bp_h,\bq}_\Omega
      \\
		&+ B_h( (\overrightarrow{\bv}, \overrightarrow{\bq}), (\bsig_{E,h}, \bsig_{V,h},\psi_h) ) 
		- B_h ( ( \overrightarrow{\bu}_h, \overrightarrow{\bp}_h ), (\btau_E, \btau_V,\varphi) ) 
		\\
	  & +\dual{ \tfrac{(k+1)^2}{h_\cF}  \jump{\overrightarrow{\bu}_{h}}, 
	  \jump{\overrightarrow{\bv}} }_{\partial \cT_h}  
     +\dual{ \tfrac{(k+1)^2}{h_\cF}\jump{\overrightarrow{\bp}_{h}}, 
	 \jump{\overrightarrow{\bq}} }_{\partial \cT_h}
	  =  \inner{\bF, \bv}_\Omega + \inner{g, \varphi}_\Omega,
	\end{split}
\end{align}
for all $\overrightarrow{\bv} = (\bv, \hat{\bv} ) \in \mathcal{U}_{h}$, $\overrightarrow{\bq} = (\bq, \hat{\bq} ) \in \mathcal{Q}_{h}$ and $(\btau_E, \btau_V, \varphi) \in  \cH_{2,h}$,  where the bilinear form $B_h(\cdot,\cdot)$ is given by 
\begin{align*}
\begin{split}
	B_h ( (\overrightarrow{\bv}, \overrightarrow{\bq} ), (\btau_E, \btau_V,\varphi) ) &:= 
	\inner*{\btau_E + \omega\btau_V -  \alpha \varphi I_d, \beps(\bv) }_{\cT_h} -  
	\inner*{\varphi, \div\bq}_{\cT_h}  
	\\
	& \qquad  - \dual*{(\btau_E + \omega\btau_V -  \alpha \varphi I_d)\bn , \jump{\overrightarrow{\bv}} }_{\partial \mathcal{T}_h} + \dual*{\varphi \bn, \jump{\overrightarrow{\bq}} }_{\partial \mathcal{T}_h}.
\end{split}
\end{align*}
Here, $\bn \in \cP_0(\partial \cT_h, \bbR^d)$ denotes the piecewise constant unit normal vector field on element boundaries, where $\bn|_{\partial K} = \bn_K$ points outward from each element $K$.
We start up problem \eqref{sd} with the initial conditions
\begin{align}\label{initial-R1-R2-h*c}
	\begin{split}
		\bu_h(0) &= \Pi^{k+1}_\cT \bu^0, \quad \hat{\bu}_h(0)  = \Pi^{k+1}_\cF (\bu^0|_{\partial \mathcal{F}_h}), \quad \bp_h(0) = \Pi^{k+1}_\cT \bp^0, \quad \hat{\bp}_h(0)  = \Pi^{k+1}_\cF (\bp^0|_{\partial \mathcal{F}_h}), 
		\\
		\bsig_{E, h}(0) &= \Pi_\cT^k \bsig_E^0, \quad \bsig_{V, h}(0) = \Pi_\cT^k \bsig_V^0, \quad \text{and} \quad \psi_h(0) = \Pi_\cT^k \psi^0,
	\end{split}
\end{align}
where the projections $\Pi_\cT^k$ and $\Pi^{k+1}_\cF$ are defined in Appendix~\ref{appendix}. 

\section{Convergence analysis of the HDG method}\label{sec:convergence}

\subsection{Well-posedness and consistency}

The following result proves that the algebraic-differential system \eqref{sd} admits a unique solution.
\begin{proposition}
	Problem~\eqref{sd}-\eqref{initial-R1-R2-h*c} is well-posed.
\end{proposition}
\begin{proof}
The system \eqref{sd} forms a set of differential-algebraic equations (DAEs). We establish its well-posedness using a standard HDG reduction technique. First, we express the numerical traces $\hat{\bu}_h$ and $\hat{\bp}_h$ in terms of the volume variables by testing \eqref{sd} with appropriately chosen test functions (specifically, setting volume test functions to zero and using arbitrary trace test functions $\hat{\bv} \in \cP_{k+1}(\cF^0_h, \mathbb{R}^d)$ and $\hat{\bq} \in \cP_{k+1}(\cF_h, \mathbb{R}^d)$). This yields an algebraic system that uniquely determines the traces. Substituting these trace expressions back into \eqref{sd} transforms the DAEs into a well-posed system of ordinary differential equations (ODEs) for the volume variables $\bu_h$, $\bp_h$, $\bsig_{E,h}$, $\bsig_{V,h}$, and $\psi_h$. We refer to \cite[Proposition 4]{meddahi2025} for a similar detailed construction.
\end{proof}	

Let us now establish the consistency of the HDG scheme \eqref{sd} with problem \eqref{eq:weakform}. The result requires additional regularity to guarantee that traces on the mesh skeleton are meaningful through the trace theorem.

\begin{proposition}\label{consistency}
Let $(\bu, \bp) \in \mathcal{C}^1_{[0,T]}(\mathcal{H}_1) \cap \mathcal{C}^0_{[0,T]}(\mathcal{X}_1)$ and $(\bsig_E, \bsig_V, \psi) \in \mathcal{C}^1_{[0,T]}(\mathcal{H}_2) \cap \mathcal{C}^0_{[0,T]}(\mathcal{X}_2)$ be the solution of \eqref{eq:IVP}. Assume that $\bsig_E + \omega\bsig_V - \alpha \psi I_d\in\cC_{[0,T]}^0(H^s(\cT_h, \bbS))$ and $(\bu,\bp)  \in \cC_{[0,T]}^0([H^s(\cT_h, \mathbb{R}^{d})]^2)$, with $s>\sfrac{1}{2}$. Then, it holds true that 
\begin{align}\label{consistent}
	\begin{split}
		&\inner*{(\dot{\bu}, \dot{\bp} ), (\bv_h,\bq_h)}_{\cH_1} + 
		\inner{(\dot{\bsig}_{E}, \dot{\bsig}_{V}, \dot \psi), (\btau_{E,h}, \btau_{V,h},\varphi_h) }_{\cH_2} + \inner{\cG\bsig_{V}, \omega\btau_{V,h}}_\Omega  + \inner{\beta \bp,\bq_h}_\Omega
      \\
		& \qquad + B_h( (\overrightarrow{\bv}_h, \overrightarrow{\bq}_h), (\bsig_{E}, \bsig_{V},\psi) ) 
		- B_h \bigg( \Big( (\bu, \bu|_{\cF_h}), (\bp, \bp|_{\cF_h}) \Big), (\btau_{E,h}, \btau_{V,h},\varphi_h) \bigg) 
		\\
	  & \qquad  +\dual{ \tfrac{(k+1)^2}{h_\cF}  (\bu- \bu|_{\cF_h}), 
	  \jump{\overrightarrow{\bv}_h} }_{\partial \cT_h}  
     +\dual{ \tfrac{(k+1)^2}{h_\cF} (\bp - \bp|_{\cF_h}), 
	 \jump{\overrightarrow{\bq}_h} }_{\partial \cT_h}
	 =  \inner{\bF, \bv_h}_\Omega + \inner{g, \varphi_h}_\Omega,
	\end{split}
\end{align}
for all $\overrightarrow{\bv}_h \in \mathcal{U}_h$, $\overrightarrow{\bq}_h \in \mathcal{Q}_h$, and $(\btau_{E,h}, \btau_{V,h} ,\varphi_h) \in  \cH_{2,h}$.
\end{proposition}
\begin{proof}
The continuity of the normal components of $\bsig -  \alpha p \mathrm{I}_d$ and $p\mathrm{I}_d$ across the interelements of $\cT_h$ gives 
\begin{align*}
	\begin{split}
			&B_h( (\overrightarrow{\bv}_h, \overrightarrow{\bq}_h), (\bsig_{E}, \bsig_{V},\psi) ) = 
			\inner*{\bsig_E + \omega\bsig_V - \alpha \psi I_d, \beps(\bv_{h})}_{\cT_h} 
			- \dual*{(\bsig_E + \omega\bsig_V  - \alpha \psi I_d)\bn, \jump{\overrightarrow{\bv}_h}  }_{\partial \mathcal{T}_h} 
			\\
			& - \inner*{\psi, \div\bq_h }_{\cT_h}  
			+ \dual*{\psi\bn, \jump{\overrightarrow{\bq}_h}  }_{\partial \mathcal{T}_h}  = \inner*{\bsig_E + \omega\bsig_V - \alpha \psi I_d, \nabla \bv_{h}}_{\cT_h} 
			- \dual*{(\bsig_E + \omega\bsig_V  - \alpha \psi I_d)\bn, \bv_h  }_{\partial \mathcal{T}_h} \\
			&\qquad - \inner*{\psi, \div\bq_h }_{\cT_h}  
			+ \dual*{\psi\bn, \bq_h  }_{\partial \mathcal{T}_h}.
	\end{split}
\end{align*}
Applying an elementwise integration by parts to the right-hand side of the previous identity, followed by the substitutions $ -\bdiv(\bsig_E + \omega\bsig_V  - \alpha \psi I_d)   = \bF   - \rho\dot{\bu }$ and $\nabla \psi    =  - \beta \bp - \chi \dot{\bp }$ yields 
\begin{align}\label{b2} 
	\begin{split}
		B_h( (\overrightarrow{\bv}_h, \overrightarrow{\bq}_h), (\bsig_{E}, \bsig_{V},\psi) ) &= - \inner{\bdiv(\bsig_E + \omega\bsig_V  - \alpha \psi I_d), \bv_{h}}_{\cT_h} + \inner{\nabla \psi, \bq_{h}}_{\cT_h} 
				\\&
		= -\inner{\rho\dot{\bu}, \bv_h }_{\Omega} -\inner{\chi\dot{\bp}, \bq_h }_{\Omega} - \inner{\beta \bp, \bq_{h}}_{\Omega} + \inner{\bF(t), \bv_{h}}_{\Omega}.	
	\end{split}
\end{align}
	
On the other hand, for all $(\btau_{E,h}, \btau_{V,h}, \varphi_h) \in \cH_{2,h}$ 
\begin{equation}\label{b1}
	B_h \bigg( \Big( (\bu, \bu|_{\cF_h}), (\bp, \bp|_{\cF_h}) \Big), (\btau_{E,h}, \btau_{V,h},\varphi_h) \bigg) = \inner*{\btau_{E,h} + \omega \btau_{V,h} - \alpha \varphi_h I_d, \beps(\bu) }_{\cT_h} -  \inner*{\varphi_h, \div\bp}_{\cT_h}.
\end{equation}
Substituting back \eqref{b2} and \eqref{b1} in the left-hand side of \eqref{consistent}  gives  
\begin{align}\label{consistent0}
	\begin{split}
		&\inner*{(\dot{\bu}, \dot{\bp} ), (\bv_h,\bq_h)}_{\cH_1} + 
		\inner{(\dot{\bsig}_{E}, \dot{\bsig}_{V}, \dot \psi), (\btau_{E,h}, \btau_{V,h},\varphi_h) }_{\cH_2} + \inner{\cG\bsig_{V}, \omega\btau_{V,h}}_\Omega  + \inner{\beta \bp,\bq_h}_\Omega
		\\
		& \qquad + B_h( (\overrightarrow{\bv}_h, \overrightarrow{\bq}_h), (\bsig_{E}, \bsig_{V},\psi) ) 
		- B_h \bigg( \Big( (\bu, \bu|_{\cF_h}), (\bp, \bp|_{\cF_h}) \Big), (\btau_{E,h}, \btau_{V,h},\varphi_h) \bigg) 
		\\
		& \qquad  +\dual{ \tfrac{(k+1)^2}{h_\cF}  (\bu - \bu|_{\cF_h}), 
		\jump{\overrightarrow{\bv}_h} }_{\partial \cT_h}  
		+\dual{ \tfrac{(k+1)^2}{h_\cF} (\bp - \bp|_{\cF_h}), 
		\jump{\overrightarrow{\bq}_h} }_{\partial \cT_h}
		\\
		&=  \inner{(\dot{\bsig}_{E}, \dot{\bsig}_{V}, \dot \psi), (\btau_{E,h}, \btau_{V,h},\varphi_h) }_{\cH_2} + \inner{\omega\cG\bsig_{V}, \btau_{V,h}}_\Omega -\inner*{\btau_{E,h} + \omega \btau_{V,h} - \alpha \varphi_h I_d, \beps(\bu) }_{\cT_h} 
		\\
		& \qquad + \inner*{\varphi_h, \div\bp}_{\cT_h}
	\end{split}
\end{align}
for all $(\btau_{E,h}, \btau_{V,h},\varphi_h)\in \mathcal H_{2,h}$ and $(\overrightarrow{\bv}_h, \overrightarrow{\bq}_h) \in \mathcal{U}_h\times \mathcal{Q}_h$. Moreover, applying Green's formula \eqref{green} in the second equation of \eqref{eq:weakform} and keeping in mind the density of the embedding $[H(\bdiv,\Omega,\bbS)]^2\times H^1(\Omega) \hookrightarrow \cH_{2}$ we obtain
\begin{equation*}
	\inner{(\dot{\bsig}_{E}, \dot{\bsig}_{V}, \dot \psi), (\btau_{E}, \btau_{V},\varphi) }_{\cH_2} = (g,\varphi)_\Omega  + \inner*{\btau_E + \omega\btau_V - \alpha \varphi I_d , \beps(\bu)}_\Omega - \inner{\omega\cG\bsig_{V}, \btau_{V}}_\Omega - \inner*{\varphi, \div \bp}_\Omega,
\end{equation*}
for all $(\btau_E,\btau_V, q) \in \cH_{2}$. Using  this identity in \eqref{consistent0} gives the result.
\end{proof}

\subsection{Convergence analysis}
The convergence analysis of the HDG method \eqref{sd} follows standard procedures. Using the stability of the HDG method and the consistency result \eqref{consistent}, we prove that the projected errors 
\begin{align*}
    \begin{alignedat}{2}
        \be_{u, h}(t)      &:= \Pi^{k+1}_\cT \bu  - \bu_{h}, &
		\quad 
		\be_{p,h}(t)       &:= \Pi^{k+1}_\cT \bp  - \bp_{h}, \\
        \hat\be_{ u, h}(t) &:= \Pi_\cF^{k+1}(\bu|_{\cF_h}) - \hat{\bu}_{h}, &\quad    
        \hat\be_{ p, h}(t) &:= \Pi_\cF^{k+1}(\bp|_{\cF_h}) - \hat{\bp}_{h}, \\
        \be_{\sigma_\bigstar,h}(t)    &:= \Pi_\cT^k\bsig_\bigstar - \bsig_{_\bigstar,h},\ \bigstar\in \set{E, V} &\quad 
        e_{\psi,h}(t)         &:= \Pi_\cT^k \psi - \psi_h, 
    \end{alignedat}
\end{align*}
can be estimated in terms of the approximation errors
\begin{align*}
    \begin{alignedat}{2}
        \bchi_{u}(t)       &:= \bu  - \Pi^{k+1}_\cT \bu, &\quad 
        \bchi_{p}(t)       &:= \bp  - \Pi^{k+1}_\cT \bp, \\
        \hat\bchi_{u}(t)  &:= \bu|_{\cF_h} - \Pi_\cF^{k+1}(\bu|_{\cF_h}), &\quad 
        \hat \bchi_{p}(t)  &:= \bp|_{\cF_h} - \Pi_\cF^{k+1}(\bp|_{\cF_h}), \\
        \bchi_{\sigma_\bigstar}(t)      &:= \bsig_\bigstar - \Pi_\cT^k\bsig_\bigstar, \ \bigstar\in \set{E, V}&\quad 
        \chi_\psi(t)           &:= \psi - \Pi_\cT^k \psi.
    \end{alignedat}
\end{align*}
In the subsequent analysis, to concisely represent both the volume errors and their corresponding trace errors, we introduce the following notation:
\begin{align*}
	\begin{alignedat}{2}
		\overrightarrow{\be}_{u,h}(t)       &:= (\be_{u,h}, \hat\be_{u,h}), &\quad 
		\overrightarrow{\be}_{p,h}(t)       &:= (\be_{p,h}, \hat\be_{p,h}), \\
		\overrightarrow{\bchi}_{u}(t)       &:= (\bchi_{u}, \hat\bchi_{u}), &\quad 
		\overrightarrow{\bchi}_{p}(t)       &:= (\bchi_{p}, \hat\bchi_{p}).
	\end{alignedat}
\end{align*}

We begin our analysis by the following boundedness property for the bilinear form $B_h(\cdot,\cdot)$.
\begin{proposition}
	There exists a constant $C>0$ independent of $h$ and $k$ such that
	\begin{equation}\label{Bhh}
		|B_h((\overrightarrow{\bv}, \overrightarrow{\bq} ),(\btau_{E,h},\btau_{V,h},\varphi_h)  )| \leq C  \norm{(\btau_{E,h},\btau_{V,h},\varphi_h) }_{\cH_2}  \abs{ (\overrightarrow{\bv}, \overrightarrow{\bq} )}_{\mathcal{U}\times \mathcal{Q}},
	\end{equation}
for all $(\btau_{E,h},\btau_{V,h},\varphi_h)  \in \mathcal H_{2,h}$ and  $(\overrightarrow{\bv}, \overrightarrow{\bq} ) \in \mathcal{U}\times \mathcal{Q}$.
\end{proposition}
\begin{proof}
Applying the Cauchy-Schwarz inequality and \eqref{normH}, we deduce that
\begin{align}\label{Bh}
	|B_h((\overrightarrow{\bv}, \overrightarrow{\bq} ),(\btau_{E},\btau_{V},\varphi)  )| &\lesssim ( \norm{(\btau_E,\btau_V,\varphi) }^2_{\cH_2} + \norm{\tfrac{h_\cF^{\sfrac{1}{2}}}{k+1} \btau_E}^2_{0,\partial \cT_h} + \norm{\tfrac{h_\cF^{\sfrac{1}{2}}}{k+1} \omega\btau_V}^2_{0,\partial \cT_h} \nonumber 
	\\
	&\quad + \norm{\tfrac{h_\cF^{\sfrac{1}{2}}}{k+1} \varphi}^2_{0,\partial \cT_h})^{\sfrac{1}{2}} \abs{ (\overrightarrow{\bv}, \overrightarrow{\bq} )}_{\mathcal{U}\times \mathcal{Q}},
\end{align}
for all $(\btau_E,\btau_V, \varphi) \in \cH_2$ such that $\btau_E, \btau_V \in H^s(\cT_h, \mathbb{R}^{d\times d}_{\text{sym}})$ and $\varphi \in H^s(\cT_h)$ with $s\geq \sfrac{1}{2}$, and for all $(\overrightarrow{\bv}, \overrightarrow{\bq} ) \in \mathcal{U}\times \mathcal{Q}$. The result follows by applying the discrete trace inequality \eqref{discTrace}.
\end{proof}

We are in a position to establish the following key stability estimate. 
\begin{lemma}\label{stab_sd}
Under the conditions of Proposition~\ref{consistency},  there exists a constant $C>0$ independent of $h$ and $k$ such that 
	\begin{align}\label{stab}
		\begin{split}
			\max_{[0, T]} &\norm{(\be_{u,h}, \be_{p,h})}^2_{\cH_1} + \max_{[0, T]} \norm{(\be_{\sigma_E,h}, \be_{\sigma_V,h},  e_{\psi,h})}^2_{\cH_2} 
			\\& 
 		 \leq C \int_0^T \Big(\norm{\tfrac{h_\cF^{\sfrac{1}{2}}}{k+1} \bchi_{\sigma_E}}^2_{0,\partial \cT_h} + \norm{\tfrac{h_\cF^{\sfrac{1}{2}}}{k+1} \bchi_{\sigma_V}}^2_{0,\partial \cT_h} +  \norm{\tfrac{h_\cF^{\sfrac{1}{2}}}{k+1} \chi_\psi}^2_{0,\partial \cT_h}  +  \abs{ (\overrightarrow{\bchi}_{u},\overrightarrow{\bchi}_{p})}^2_{\mathcal{U} \times \hat{\mathcal{U}}} \Big) \, \text{d}t.
		\end{split}
	\end{align} 
\end{lemma}
\begin{proof}
By virtue of the orthogonality properties
\begin{align*}
    \inner*{(\dot{\bchi}_{u}, \dot{\bchi}_{p}), (\bv, \bq)}_{\cH_1} + \inner{\beta \bchi_{p},\bq}_\Omega &=0 \quad \forall (\bv, \bq) \in \cH_{1,h}, 
    \\
    \inner{(\dot{\bchi}_{\sigma_E}, \dot{\bchi}_{\sigma_V}, \dot{\bchi}_{\psi}), (\btau_E, \btau_V,\varphi) }_{\cH_2} + \inner{\cG\bchi_{\sigma_V}, \omega\btau_V}_\Omega   &=0 \quad \forall (\btau_E, \btau_V,\varphi)\in \cH_{2,h},
\end{align*}
and the consistency result \eqref{consistent}, we can establish that for all $(\btau_E, \btau_V,\varphi) \in \mathcal H_{2,h}$ and $(\overrightarrow{\bv}, \overrightarrow{\bq} )\in \mathcal{U}_h \times \mathcal{Q}_h$:
\begin{align}\label{orthog} 
    \begin{split}
        & \inner*{(\dot{\be}_{u,h}, \dot{\be}_{p,h}), (\bv, \bq)}_{\cH_1} + 
        \inner{(\dot{\be}_{\sigma_E,h}, \dot{\be}_{\sigma_V,h}, \dot{e}_{\psi,h}), (\btau_E, \btau_V,\varphi) }_{\cH_2}  + \inner{\beta \be_{p,h},\bq}_\Omega + 
        \inner{\cG\be_{\sigma_V,h}, \omega\btau_V}_\Omega
        \\
        & \qquad + B_h((\overrightarrow{\bv}, \overrightarrow{\bq} ), (\be_{\sigma_E,h}, \be_{\sigma_V,h}, e_{\psi,h}) ) 
        - B_h((\overrightarrow{\be}_{u,h}, \overrightarrow{\be}_{p,h} ), (\btau_E, \btau_V, \varphi) )
        \\ 
        & \qquad + \dual{ \tfrac{(k+1)^2}{h_\cF}\jump{\overrightarrow{\be}_{u,h}}, \jump{\overrightarrow{\bv}}}_{\partial \cT_h}
        +\dual{ \tfrac{(k+1)^2}{h_\cF}\jump{\overrightarrow{\be}_{p,h}}, \jump{\overrightarrow{\bq}}}_{\partial \cT_h}
        \\ &
        = \dual{(\bchi_{\sigma_E} + \omega \bchi_{\sigma_V} - \alpha \chi_\psi I_d)\bn, \jump{\overrightarrow{\bv}}}_{\partial \cT_h} - \dual{\chi_\psi\bn, \jump{\overrightarrow{\bq}}}_{\partial \cT_h}
        \\ &
        \qquad + B_h((\overrightarrow{\bchi}_{u}, \overrightarrow{\bchi}_{p} ), (\btau_E, \btau_V, \varphi) ) 
        -\dual{ \tfrac{(k+1)^2}{h_\cF}\jump{\overrightarrow{\bchi}_{u}}, \jump{\overrightarrow{\bv}}}_{\partial \cT_h}
        - \dual{ \tfrac{(k+1)^2}{h_\cF}\jump{\overrightarrow{\bchi}_{p}}, \jump{\overrightarrow{\bq}}}_{\partial \cT_h}.
    \end{split} 
\end{align}

This identity follows from the fact that for all $(\underline{\bv} , \hat{\underline{\bv}} )\in \cH_{1,h} \times \hat{\cH}_{1,h}$:
\[
    B_h((\overrightarrow{\bv}, \overrightarrow{\bq} ), (\bchi_{\sigma_E}, \bchi_{\sigma_V}, \chi_{\psi}) ) = -\dual{(\bchi_{\sigma_E} + \omega \bchi_{\sigma_V} - \alpha \chi_\psi I_d)\bn, \jump{\overrightarrow{\bv}}}_{\partial \cT_h} + \dual{\chi_\psi\bn, \jump{\overrightarrow{\bq}}}_{\partial \cT_h},
\]
which is supported by the polynomial space inclusions:
\[
    \beps(\cP_{k+1}(\cT_h, \bbR^d)) \subset \cP_k(\cT_h, \bbS) \quad  \text{and} \quad  \div(\cP_{k+1}(\cT_h, \bbR^d)) \subset \cP_k(\cT_h).
\]
Choosing  $\btau_E = \be_{\sigma_E,h}$, $\btau_E = \be_{\sigma_V,h}$, $\varphi = e_{\psi,h}$, and $(\overrightarrow{\bv} ,\overrightarrow{\bq} ) = (\overrightarrow{\be}_{u,h}, \overrightarrow{\be}_{p,h} )$ in \eqref{orthog} and applying the Cauchy-Schwarz inequality together with \eqref{Bhh} yield 
\begin{align*}
	\frac12 \frac{\text{d}}{\text{d}t} &\big\{\norm{ (\be_{u,h}, \be_{p,h})}^2_{\cH_1} 
	+ \norm{(\be_{\sigma_E,h}, \be_{\sigma_V,h}, e_{\psi,h})}^2_{\cH_2}  \big\} 
	+ \norm{ \beta^{\sfrac12}  \be_{p,h}}^2_{0,\Omega} + \inner{ \omega \cG  \be_{\bsig_V,h}, \be_{\bsig_V,h}}_{0,\Omega} 
	\\
	& \qquad + \norm{\tfrac{k+1}{h_\cF^{\sfrac{1}{2}}}\jump{\overrightarrow{\be}_{u,h}}}^2_{0,\partial \cT_h} + \norm{\tfrac{k+1}{h_\cF^{\sfrac{1}{2}}}\jump{\overrightarrow{\be}_{p,h}}}^2_{0,\partial \cT_h}  
	 \\
	&\leq \norm{\tfrac{h_\cF^{\sfrac{1}{2}}}{k+1} (\bchi_{\sigma_E} + \omega \bchi_{\sigma_V} - \alpha \chi_\psi I_d)\bn}_{0,\partial \cT_h} 
	\norm{\tfrac{k+1}{h_\cF^{\sfrac{1}{2}}} \jump{\overrightarrow{\be}_{u,h}}}_{0,\partial \cT_h} + \norm{\tfrac{h_\cF^{\sfrac{1}{2}}}{k+1} \chi_\psi}_{0,\partial \cT_h} 
	\norm{\tfrac{k+1}{h_\cF^{\sfrac{1}{2}}} \jump{\overrightarrow{\be}_{p,h}}}_{0,\partial \cT_h} 
	\\
	&\quad + C \norm{(\be_{\sigma_E,h}, \be_{\sigma_V,h}, e_{\psi,h})}_{\cH_2}  
	\abs{ (\overrightarrow{\bchi}_{u}, \overrightarrow{\bchi}_{p} )}_{\mathcal{U} \times \mathcal{Q}} 
	\\
	&\quad + \norm{\tfrac{k+1}{h_\cF^{\sfrac{1}{2}}} \jump{\overrightarrow{\bchi}_{u}}  }_{0,\partial \cT_h} 
	\norm{\tfrac{k+1}{h_\cF^{\sfrac{1}{2}}} \jump{\overrightarrow{\be}_{u,h}}  }_{0,\partial \cT_h} + \norm{\tfrac{k+1}{h_\cF^{\sfrac{1}{2}}} \jump{\overrightarrow{\bchi}_{p}}  }_{0,\partial \cT_h} 
	\norm{\tfrac{k+1}{h_\cF^{\sfrac{1}{2}}} \jump{\overrightarrow{\be}_{p,h}}  }_{0,\partial \cT_h}. 
\end{align*}
We notice that, because of assumption \eqref{initial-R1-R2-h*c}, the projected errors  satisfy vanishing initial conditions, namely, $\be_{\sigma_E,h}(0) = \be_{\sigma_V,h}(0)= \mathbf 0$, $e_{\psi,h}(0) = 0$  and $\overrightarrow{\be}_{u,h} (0) = \overrightarrow{\be}_{p,h}(0) = (\mathbf 0, \mathbf 0)$. Hence, integrating over $t\in (0, T]$ and using again the Cauchy-Schwarz inequality we deduce that 
\begin{align*}
 		&\norm{ (\be_{u,h}, \be_{p,h})}^2_{\cH_1} 
		 + \norm{(\be_{\sigma_E,h}, \be_{\sigma_V,h}, \be_{\psi,h})}^2_{\cH_2} +  \int_0^t \norm{\tfrac{k+1}{h_\cF^{\sfrac{1}{2}}}\jump{\overrightarrow{\be}_{u,h}}}^2_{0,\partial \cT_h}\,\text{d}s + \int_0^t \norm{\tfrac{k+1}{h_\cF^{\sfrac{1}{2}}}\jump{\overrightarrow{\be}_{p,h}}}^2_{0,\partial \cT_h}\,\text{d}s
 		\\
 		&\lesssim 
		\Big(\int_0^T (\norm{\tfrac{h_\cF^{\sfrac{1}{2}}}{k+1} (\bchi_{\sigma_E} + \omega \bchi_{\sigma_V} - \alpha \chi_\psi I_d)\bn}^2_{0,\partial \cT_h}  +  \norm{\tfrac{h_\cF^{\sfrac{1}{2}}}{k+1} \chi_\psi }^2_{0,\partial \cT_h} + \norm{\tfrac{k+1}{h_\cF^{\sfrac{1}{2}}} \jump{\overrightarrow{\bchi}_{u}}  }^2_{0,\partial \cT_h}
		\\
		&\qquad \qquad  + \norm{\tfrac{k+1}{h_\cF^{\sfrac{1}{2}}} \jump{\overrightarrow{\bchi}_{p}}  }^2_{0,\partial \cT_h})  \text{d}t\Big)^{\sfrac{1}{2}}  
		\times 
 		\Big( \int_0^T( \norm{\tfrac{k+1}{h_\cF^{\sfrac{1}{2}}} \jump{\overrightarrow{\be}_{u,h}}  }^2_{0,\partial \cT_h} + \norm{\tfrac{k+1}{h_\cF^{\sfrac{1}{2}}} \jump{\overrightarrow{\be}_{p,h}}  }^2_{0,\partial \cT_h})\ \text{d}t\Big)^{\sfrac{1}{2}}
 		\\
 		&\qquad \qquad + \Big(\int_{0}^{T}\norm{(\be_{\sigma_E,h}, \be_{\sigma_V,h}, e_{\psi,h})}^2_{\cH_2} \text{d}t\Big)^{\sfrac12} \Big( \int_0^T \abs{ (\overrightarrow{\bchi}_{u}, \overrightarrow{\bchi}_{p} )}^2_{\mathcal{U} \times \mathcal{Q}}  \text{d}t \Big)^{\sfrac12},\quad \forall t\in (0, T].
\end{align*}
Finally, a simple application of Young's inequality yields 
\begin{align*}
 		 &\max_{[0, T]}\norm{ (\be_{u,h}, \be_{p,h})}^2_{\cH_1} + \max_{[0, T]} \norm{(\be_{\sigma_E,h}, \be_{\sigma_V,h}, e_{\psi,h})}^2_{\cH_2} +  \int_0^T (\norm{\tfrac{k+1}{h_\cF^{\sfrac{1}{2}}}\jump{\overrightarrow{\be}_{u,h}}}^2_{0,\partial \cT_h} + \norm{\tfrac{k+1}{h_\cF^{\sfrac{1}{2}}}\jump{\overrightarrow{\be}_{p,h}}}^2_{0,\partial \cT_h})\text{d}t
 		\\
 		&\qquad \qquad \lesssim \int_0^T \Big(\norm{\tfrac{h_\cF^{\sfrac{1}{2}}}{k+1} (\bchi_{\sigma_E} + \omega \bchi_{\sigma_V} - \alpha \chi_\psi I_d)}^2_{0,\partial \cT_h}  +  \norm{\tfrac{h_\cF^{\sfrac{1}{2}}}{k+1} \chi_\psi }^2_{0,\partial \cT_h}  +  \abs{ (\overrightarrow{\bchi}_{u}, \overrightarrow{\bchi}_{p} )}^2_{\mathcal{U} \times \mathcal{Q}} \Big) \, \text{d}t,
\end{align*}
and the result follows.
\end{proof}

As a direct consequence of the stability estimate \eqref{stab}, we obtain the following convergence result for the HDG method \eqref{sd}-\eqref{initial-R1-R2-h*c}.
\begin{theorem}\label{hpConv}
	Let $(\bu, \bp) \in \mathcal{C}^1_{[0,T]}(\mathcal{H}_1) \cap \mathcal{C}^0_{[0,T]}(\mathcal{X}_1)$ and $(\bsig_E, \bsig_V, \psi) \in \mathcal{C}^1_{[0,T]}(\mathcal{H}_2) \cap \mathcal{C}^0_{[0,T]}(\mathcal{X}_2)$ be the solution of \eqref{eq:IVP}. Assume that $\bsig_E, \bsig_V \in \cC^0_{[0,T]}(H^{1+r}(\Omega, \mathbb{R}^{d\times d}_{\text{sym}}))$, $\psi \in \cC^0_{[0,T]}(H^{1+r}(\Omega))$ and $\bu, \bp  \in \cC^0_{[0,T]}(H^{2+r}( \Omega, \bbR^{d}))$, with $r\geq 0$. Then, there exists a constant $C>0$ independent of $h$ and $k$ such that 
\begin{align*}
		 &\max_{[0, T]}\norm{ (\bu-\bu_h, \bp-\bp_h)(t))}_{\cH_1} + \max_{[0, T]} \norm{(\bsig_E - \bsig_{E,h}, \bsig_V - \bsig_{V,h}, \psi - \psi_h)(t)}_{\cH_2} 
		\\
		&\, \leq  C \tfrac{h_K^{\min\{ r, k \}+1}}{(k+1)^{r+\sfrac12}} \Big( \max_{[0,T]}\norm{\bu }_{2+r,\Omega} + \max_{[0,T]}\norm{\bp }_{2+r,\Omega} + \max_{[0, T]}\norm*{\bsig_E}_{1+r, \Omega} + \max_{[0, T]}\norm*{\bsig_V}_{1+r, \Omega} + \max_{[0, T]}\norm*{\psi}_{1+r, \Omega}   \Big). 
\end{align*}
\end{theorem}
\begin{proof}
	By applying the triangle inequality and the stability estimate \eqref{stab}, we obtain: 
	\begin{align*}
		&\max_{[0, T]}\norm{ (\bu-\bu_h, \bp-\bp_h)(t))}_{\cH_1} + \max_{[0, T]} \norm{(\bsig_E - \bsig_{E,h}, \bsig_V - \bsig_{V,h}, \psi - \psi_h)(t)}_{\cH_2} 
		\\
		&\lesssim \max_{[0, T]}\norm{(\bchi_u, \bchi_p)(t)}^2_{\cH_1} 
		+ \max_{[0, T]}\norm{(\bchi_{\sigma_E}, \bchi_{\sigma_V}, \chi_\psi)(t)}^2_{\cH_2} 
		\\
		&\quad + \left( \int_0^T \Big(\norm{\tfrac{h_\cF^{\sfrac{1}{2}}}{k+1} \bchi_{\sigma_E}}^2_{0,\partial \cT_h} + \norm{\tfrac{h_\cF^{\sfrac{1}{2}}}{k+1} \bchi_{\sigma_V} }^2_{0,\partial \cT_h}  +  \norm{\tfrac{h_\cF^{\sfrac{1}{2}}}{k+1} \chi_\psi }^2_{0,\partial \cT_h}  +  \abs{ (\overrightarrow{\bchi}_{u}, \overrightarrow{\bchi}_{p} )}^2_{\mathcal{U} \times \mathcal{Q}} \Big) \, \text{d}t \right)^{\sfrac{1}{2}}.
	\end{align*}
	The desired result follows directly from the error estimates \eqref{tool1} and \eqref{tool2}.
\end{proof}
	
\begin{remark}\label{R1}
	The convergence analysis in Theorem~\ref{hpConv} leads to two main observations:
	\begin{itemize}
	\item The theoretical bound is sub-optimal by a factor \((k+1)^{1/2}\).  This slight deterioration arises from the combination of the discrete trace-inverse inequality \eqref{discTrace} with the standard stabilization weight \((k+1)^2/h_{\cF}\) introduced to ensure stability; together they lead to estimate \eqref{tool2}.
	\item The numerical experiments of Subsection~\ref{example1} indicate optimal convergence rates for $\bu$ and $\bp$ in the $L^2$-norms, suggesting a gap between theoretical predictions and practical performance that will be addressed in future work.
	\end{itemize}
\end{remark}

\section{Numerical results}\label{sec:numresults}

The numerical results presented in this section have been implemented using the finite element library \texttt{Netgen/NGSolve} \cite{schoberl2014c++}. Firstly, we confirm the accuracy of our HDG scheme by analyzing a problem with a manufactured solution. Then, we consider a practical model problem inspired by \cite{antoniettiIMA, morency}.

\subsection{Validation of the convergence rates}\label{example1}

In this example, we confirm the decay of error as predicted by Theorem~\ref{hpConv} with respect to the parameters $h$ and $k$.  We employ successive levels of refinement on an unstructured mesh and compare the computed solutions to an exact solution of problem \eqref{eq:poro_thermo_visco} given by 
\begin{align}\label{exactSol}
\begin{split}
\psi(x,y,t) &:= sin(\pi x) \sin(\pi y) \cos(2\pi t) \quad \text{in $\Omega\times (0, T]$},
\\
\bu(x,y,t) &:= \begin{pmatrix}
	2\pi y  \sin(\pi x) \cos(\pi y)  \cos(2\pi t)
	\\
	2\pi x  \cos(\pi x) \sin(\pi y) \cos(2\pi t)
	\end{pmatrix} \quad \text{in $\Omega\times (0, T]$},
\end{split}
\end{align}
where $\Omega = (0,1)\times (0, 1)$. We assume that the medium characterized by \eqref{eq:ptv-b}-\eqref{eq:ptv-c}  is isotropic in the sense that the fourth-order elastic and viscoelastic tensors are defined by 
\begin{equation}\label{eq:hooke} 
\cC \btau = 2 \mu_{\cC}\boldsymbol{\zeta} + \lambda_{\cC} \tr(\boldsymbol{\zeta}) I 
\qquad \cD \boldsymbol{\zeta} = 2 \mu_{\cD} \btau +  \lambda_{\cD} \tr(\boldsymbol{\zeta}) I,
\end{equation}
in terms of Lamé coefficients $\mu_{\cD}>\mu_{\cC}>0$ and $\lambda_{\cD}>\lambda_{\cC}>0$. 

We compute the source terms $\bF$ and $g$ corresponding to the manufactured solution \eqref{exactSol} of \eqref{eq:poro_thermo_visco} with the material parameters given by \eqref{L1} or \eqref{L2}. We prescribe non-homogeneous boundary conditions \eqref{eq:BC} with $\Gamma^s_D = \partial \Omega$ and $\Gamma^f_N=\partial \Omega$.

In our first test, the parameters are chosen as 
\begin{gather}\label{L1}
	\begin{split}
		\rho = 1, \quad \mu_{\cC} = 10,\quad  \lambda_{\cC} = 30,\quad   \mu_{\cD} = 20,\quad   \lambda_{\cD} = 40, 
	\\
	\alpha = 1, \quad \beta = 1, \quad \omega=1,\quad \chi = 1, \quad s = 1.
	\end{split}
\end{gather}

We partition the time interval  $[0, T]$ into uniform subintervals of length $\Delta t$. We discretize in time using the Crank-Nicolson method. Owing to its second-order accuracy in time, we choose the time step such that $\Delta t \approx O(h^{(k+2)/2})$. This ensures that the temporal discretization error remains asymptotically smaller than the spatial error, thereby preserving the expected spatial convergence rates. To verify the accuracy of our method, we evaluate the following $L^2$-norm error measures at the final time $T$:  
\begin{align}\label{Errors1}
	\begin{split}
		\mathtt{e}^{L}_{hk}(\bsig,\psi) &:= \norm{(\bsig_E(T) - \bsig_{E,h}^L, \bsig_V(T) - \bsig_{E,h}^L, \psi(T) - \psi_h^L) }_{\cH_2}
 \\
 \mathtt{e}^L_{hk}(\bu, \bp)  &:= \norm{(\bu(T)  - \bu^L_h, \bp(T)  - \bp^L_h)}_{\cH_1}.
	\end{split}
\end{align}
 
In Figure~\ref{fig1}, we present the errors as functions of the mesh size $h$ for three different polynomial degrees $k$. The $L^2$-errors \eqref{Errors1} are displayed in log-log plots, with the expected rates of convergence represented by dashed lines. The results show that the errors $\mathtt{e}^{L}_{hk}(\bsig,\psi)$ and $\mathtt{e}^L_{hk}(\bu, \bp)$ achieve the optimal convergence rates of $O(h^{k+1})$ and $O(h^{k+2})$, respectively. 

\begin{figure}[!ht]
	\centering
	\includegraphics[width=\textwidth, height=0.2\textheight]{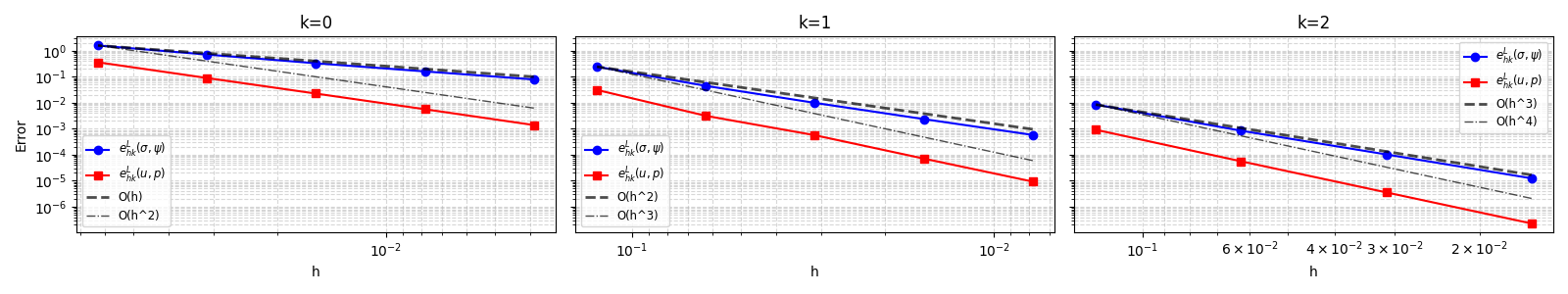}
	\caption{The errors \eqref{Errors1} are plotted against the mesh size h for various polynomial degrees k, using temporal over-refinements. Problem \eqref{eq:poro_thermo_visco} is defined with coefficients \eqref{L1} and the exact solution \eqref{exactSol}. }
	\label{fig1}
\end{figure}

To verify the accuracy and stability of the scheme for nearly incompressible poroelastic media with small storage coefficient $s$, we repeat the same experiment using the following set of material parameters:
\begin{gather}\label{L2}
	\begin{split}
		\rho = 1, \quad  \nu_{\cC} = 0.49,\quad   E_{\cC} = 100,\quad \nu_{\cD} = 0.4999,\quad   E_{\cD} = 1000, 
	\\
	\alpha = 1, \quad \beta = 1, \quad \omega=1,\quad \chi = 1, \quad s = 10^{-6}.
	\end{split}
\end{gather}
We recall that the Lamé coefficients are related to the Young's modulus $E$ and Poisson's ratio $\nu$ by the relations $\lambda_{\star} = \frac{E_{\star} \nu_{\star}}{(1+\nu_{\star})(1-2\nu_{\star})}$ and $\mu_{\star} = \frac{E_{\star}}{2(1+\nu_{\star})}$, where $\star \in \set{\cC, \cD}$.

The error decay for this case is shown in Figure~\ref{fig2}. These results demonstrate the robustness of the proposed HDG scheme in producing accurate approximations for the chosen extreme parameter values. The results are consistent with those obtained in Figure~\ref{fig1} for the coefficients in \eqref{L1}.

\begin{figure}[!ht]
	\centering
	\includegraphics[width=\textwidth, height=0.2\textheight]{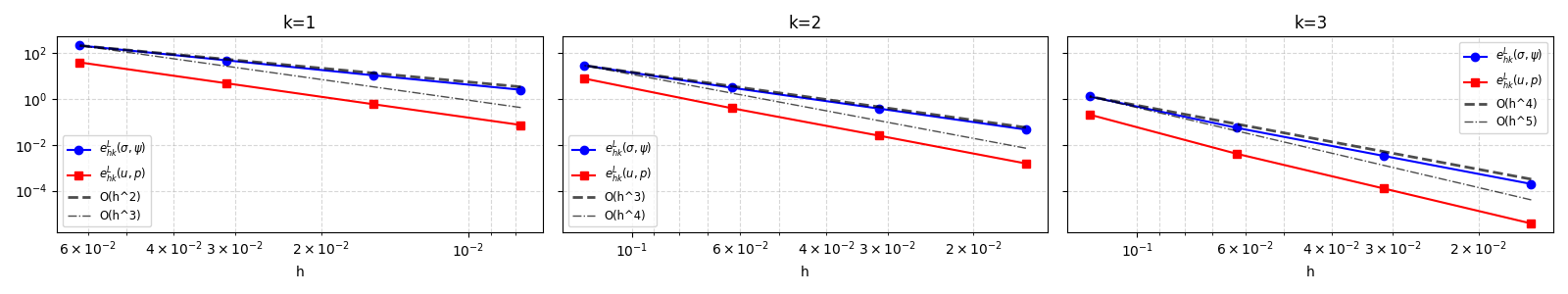}
	\caption{The errors \eqref{Errors1} are plotted against the mesh size h for various polynomial degrees k, using temporal over-refinements. Problem \eqref{eq:poro_thermo_visco} is defined with coefficients \eqref{L2} and the exact solution \eqref{exactSol}.}
	\label{fig2}
\end{figure}

Having established the scheme's behavior with respect to the mesh size $h$, we now turn to testing its performance with respect to the parameters $k$ and $\Delta t$.

We fix the space mesh size $h = 1/4$ and the time step $\Delta t = 10^{-6}$ and let the polynomial degree $k$ vary from 1 to 7. In Figure~\ref{T3} we present the errors $\mathtt{e}^{L}_{hk}(\bsig,p)$ and $\mathtt{e}^L_{hk}(\bu, \bp)$ at  $t = 0.3$  plotted against the polynomial degree $k$ on a semi-logarithmic scale. This example uses the same manufactured solution derived from \eqref{exactSol} with material coefficients given in \eqref{L1}. As expected, exponential convergence is observed.

In Figure~\ref{T4}, we show the convergence results obtained by fixing the spatial mesh size at $h=1/16$ and the polynomial degree at $k=3$, while varying the time step $\Delta t$ used to uniformly subdivide the time interval $[0,T]$ with $T=0.5$. The expected convergence rate of $O(\Delta t^2)$ is achieved as the time step is refined.

\begin{table}[!ht]
	\begin{minipage}{0.47\linewidth}
		\centering
  \includegraphics[width=0.9\textwidth]{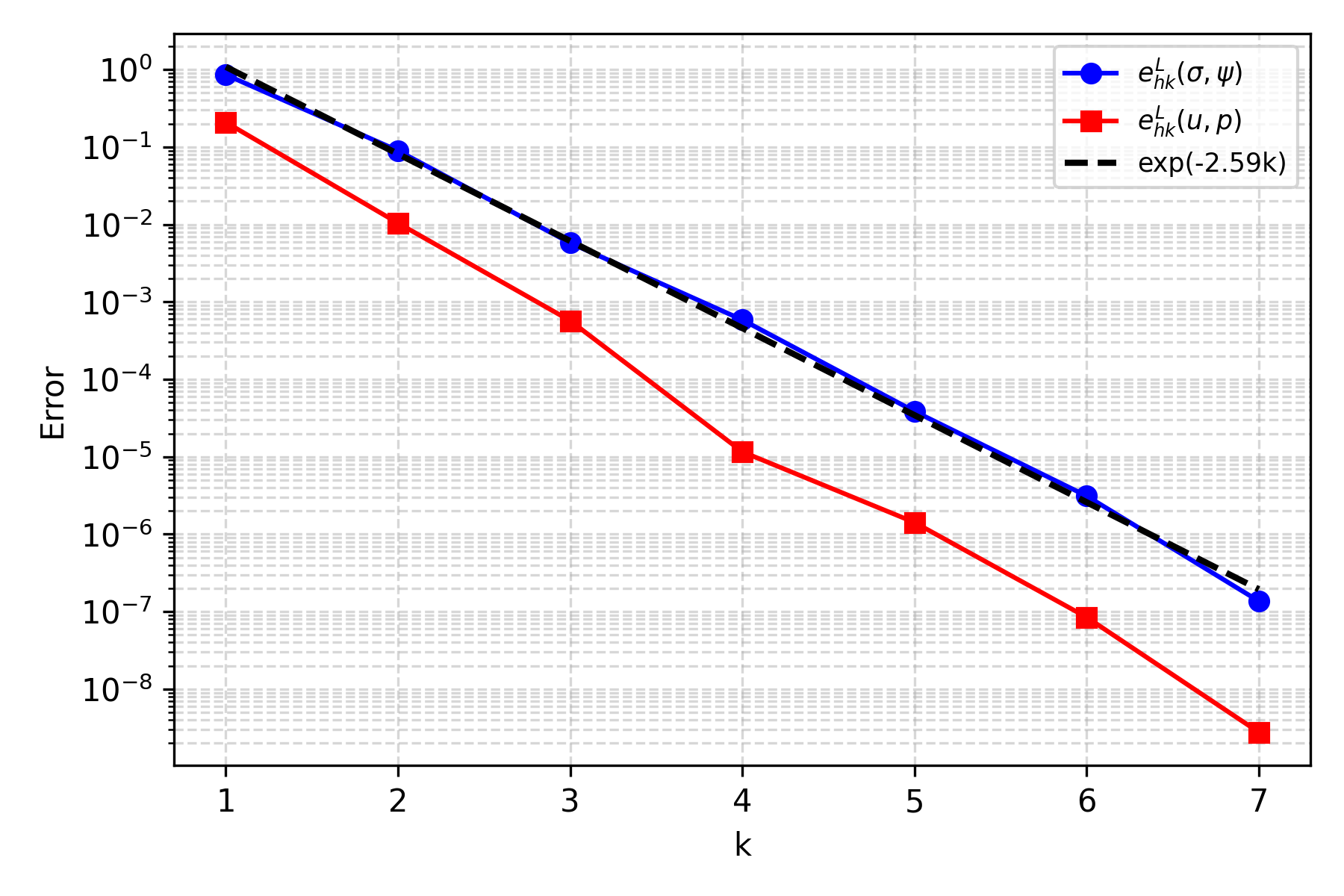}
		\captionof{figure}{Computed errors versus the polynomial degree $k$ with $h=1/4$ and $\Delta t = 10^{-6}$. The errors are measured at $t=0.3$, by employing the coefficients \eqref{L1}.}
		\label{T3}
	  \end{minipage}
	\hfill
	\begin{minipage}{0.47\linewidth}
	  \centering
\includegraphics[width=0.9\textwidth]{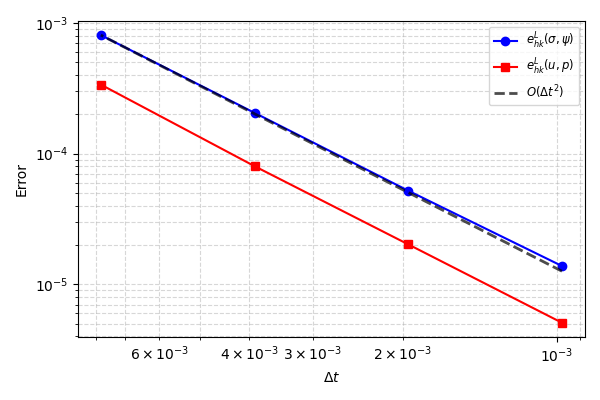}
	  \captionof{figure}{Computed errors for a sequence of uniform refinements in time with $h=1/16$ and $k=3$. The errors are measured at $t=0.5$, with the coefficients \eqref{L1}. }
	  \label{T4}
	\end{minipage}
  \end{table} 

  \subsection{Wave propagation in a thermoelastic medium}

  \begin{figure}[!ht]
	\begin{center}
	\includegraphics[scale=0.3]{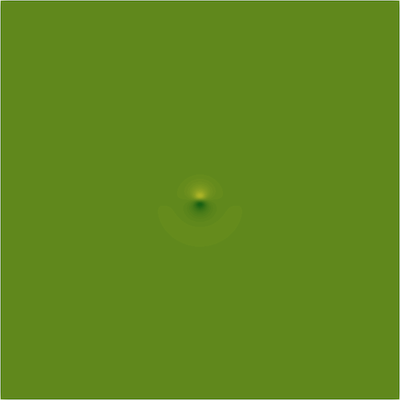}
	\includegraphics[scale=0.3]{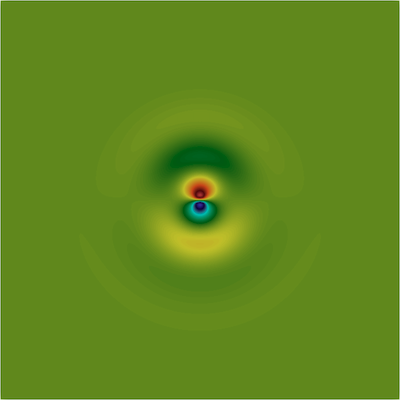}
	\includegraphics[scale=0.3]{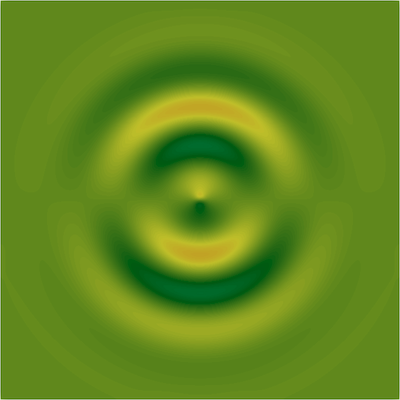}
	\\
	\includegraphics[width=0.5\textwidth]{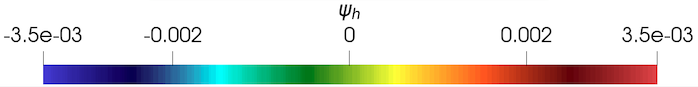}
	\\
	\includegraphics[scale=0.3]{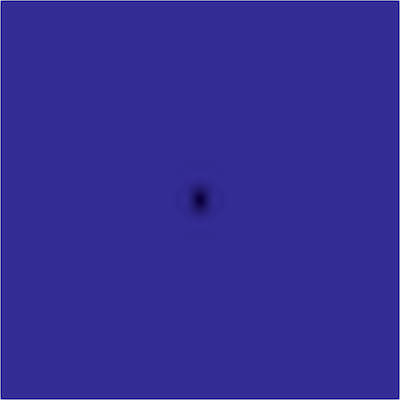}
	\includegraphics[scale=0.3]{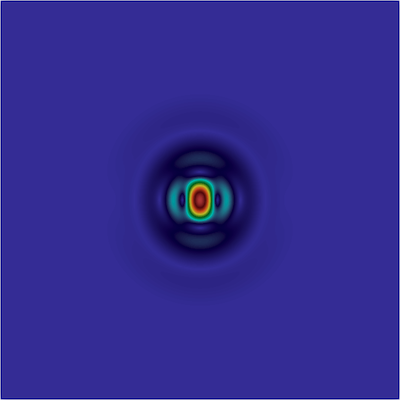}
	\includegraphics[scale=0.3]{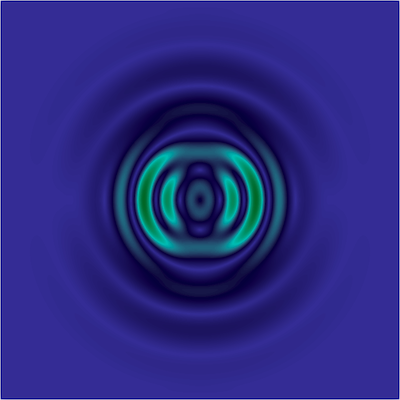}
	\\
	\includegraphics[width=0.5\textwidth]{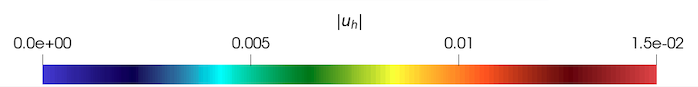}
	\end{center}
	\vspace{-4mm}
	\caption{
	Snapshots of the temperature $\psi$ (top row) and the module of the solid velocity $\bu$ (bottom row) at times 0.1\,\unit{s}, 0.3\,\unit{s}, and 0.5\,\unit{s} (left to right panels). Problem \eqref{eq:poro_thermo_visco} is solved using the source terms \eqref{sourceTE}, with vanishing initial and Dirichlet boundary conditions, and parameter set \eqref{L3}, with $h=50$, $k=5$, and $\Delta t = 10^{-5}$.}
	\label{fig:thermoElast}
	\end{figure}

In this test, we reproduce the numerical simulation from \cite{bonetti2024} of thermal-wave propagation in a homogeneous, isotropic thermoelastic medium occupying the square domain $\Omega = (0, 2310)\times (0, 2310)\, \unit{m^2}$. A vertically oriented force source is imposed at the center of the domain:
\begin{equation}\label{sourceTE}
	\bF = \exp\Bigl(-\tfrac{9\|\mathbf{r}\|^2}{2h^2}\Bigr)\;S(t)\;\mathbf{e}_y
\end{equation}
where $\boldsymbol{r} := (x- 1155, y - 1155)^\texttt{t}$ measures the distance from $(1155, 1155)$, and  
\begin{equation}\label{sourceTime}
	S(t)\;=\;A_0\cos\bigl(2\pi f_0 (t-t_0)\bigr)\,\exp\bigl(-2f_0^2\,(t-t_0)^2\bigr)
\end{equation}
models a compressive seismic pulse centered at \(t_0=0.3\)\,\unit{s}, with amplitude $A_0 = 10^4$ and peak-frequency $f_0 = 5$ \unit{Hz}. The physical parameters are set to 
\begin{gather}\label{L3}
	  \begin{split}
		\rho = 2650\, \unit{kg\per m^3 }, \quad   \mu_{\cC} = 6\cdot 10^9\, \unit{Pa},\quad  \lambda_{\cC} = 4 \cdot 10^9\, \unit{Pa}, 
	  \\
	  \chi = 1.49\cdot  10^{-8}/\theta \quad \text{and} \quad  \beta = 1/\theta,\quad \text{with} \quad \theta = 10.5\, \unit{m^2 Pa \per K^2 s },
	  \\
	  \alpha = 79200\, \unit{Pa\per K },\quad s = 117\, \unit{Pa\per K^2 }.
	  \end{split}
\end{gather}

In Figure~\ref{fig:thermoElast} we observe two coupled phenomena developing from the vertical source at the center of the domain.  In the top row (temperature $\psi$) a diffusive thermal front (the so-called $T$-wave) detaches radially from the source and propagates isotropically, with its amplitude decaying due to thermal conduction.  In the bottom row (velocity magnitude $|\bu|$) an elastic wavefront initially appears as a single nearly circular ring. By $t=0.5\,$s this ring has broadened and split into two concentric fronts: the outer, faster front corresponds to the compressional ($P$-) wave, and the inner, slower one to the shear ($S$-) wave. The overall symmetry of the patterns and the gradual emergence of two elastic modes are in excellent agreement with the results reported in \cite{bonetti2024, carcione2019}.

\subsection{Wave propagation in a poro-viscoelastic medium}

\begin{figure}[!ht]
    \begin{center}
    \includegraphics[scale=0.3]{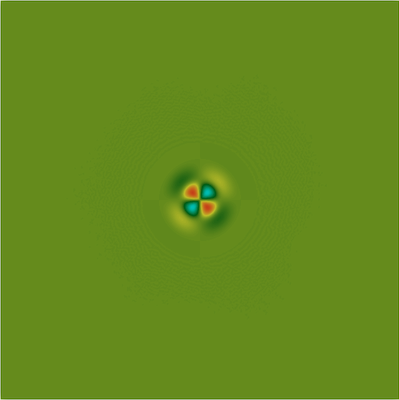}
    \includegraphics[scale=0.3]{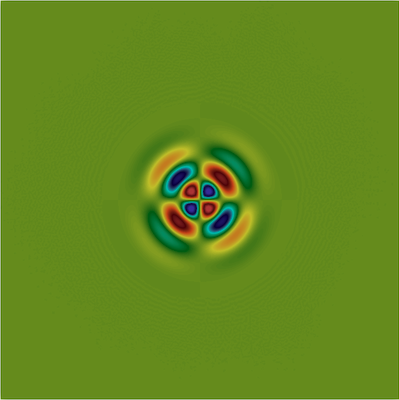}
    \includegraphics[scale=0.3]{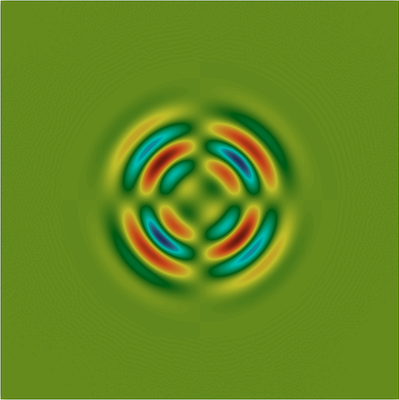}
    \\
    \includegraphics[scale=0.3]{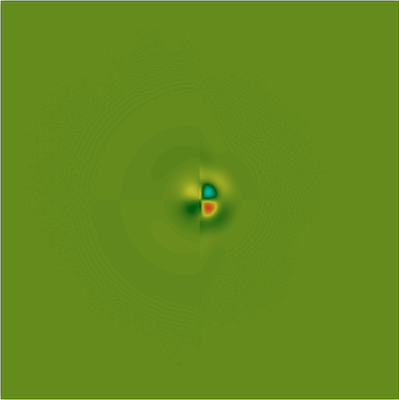}
    \includegraphics[scale=0.3]{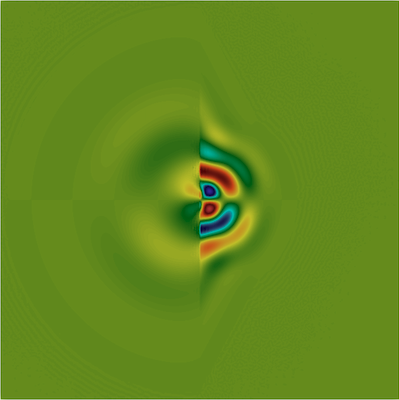}
    \includegraphics[scale=0.3]{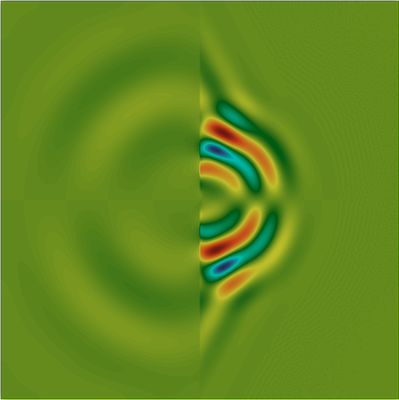}
    \\
    \includegraphics[width=0.5\textwidth]{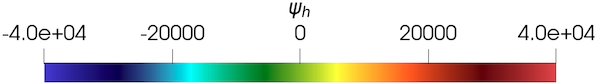}
    \end{center}
    \vspace{-4mm}
    \caption{
    Snapshots of the pressure field $\psi$ at $t = 0.2\,\unit{s}$, $0.4\,\unit{s}$, and $0.6\,\unit{s}$ (left to right). Problem \eqref{eq:poro_thermo_visco} is solved using the source term \eqref{Fshear}, homogeneous initial and Dirichlet boundary conditions, and parameter set \eqref{L4}, with mesh size $h=50$, polynomial degree $k=5$, and time step $\Delta t = 10^{-5}\,\unit{s}$. Top row: purely poroelastic medium ($\omega = 0$ throughout $\Omega$); bottom row: heterogeneous medium with poro-viscoelastic properties ($\omega = 0.9$) in the left half of $\Omega$ and poroelastic properties in the right half.}
    \label{fig:poroElast1}
\end{figure}

To evaluate the efficacy of our approach in modeling composite materials with heterogeneous viscoelastic properties, we present an illustrative example inspired by \cite{meddahi2023hp}. This numerical experiment examines wave dynamics within a hybrid medium comprising both poro-viscoelastic and purely poroelastic regions.

We consider a wave propagation problem in a computational domain $\Omega = (-1155, 1155)\times (0, 2310)\, \unit{m^2}$ containing an explosive source positioned at $\boldsymbol{x}_s = (0,1155) \, \unit{m}$. The spatial source function is defined as:
\[
  \boldsymbol{b}(x,y) = \begin{cases}
        \big(1 - \frac{\norm{\boldsymbol{r}}^2}{4h^2}\big)\frac{\boldsymbol{r}}{\norm{\boldsymbol{r}}} & \text{if $\norm{\boldsymbol{r}} < 2h$}
        \\
        \boldsymbol{0} & \text{otherwise}
    \end{cases},
\]
where $\boldsymbol{r} = (x, y - 1155)^\texttt{t}$ represents the distance vector from the source position, and $h$ denotes the mesh size used for spatial discretization. The forcing term is formulated as: 
\begin{equation}\label{Fshear}
    \bF = \mathbf{M} \boldsymbol{b}(x,y) S(t),
\end{equation} 
where $S(t)$ is defined by \eqref{sourceTime} and $\mathbf{M}$ is a matrix with entries $M_{11} = M_{22} = 0$ and $M_{12} = M_{21} = \sfrac{1}{\sqrt{2}}$. This configuration generates a shear source centered at position $\boldsymbol{x}_s$.

The physical parameters employed in this simulation are as follows: 
\begin{gather}\label{L4}
    \begin{split}
        \mu_{\cC} = 10^9\,\unit{Pa},\quad \lambda_{\cC} = 4 \cdot 10^8\,\unit{Pa},\quad \mu_{\cD} = 4\cdot 10^9\,\unit{Pa},\quad \lambda_{\cD} = 7\cdot 10^9\,\unit{Pa}, 
    \\
    \rho_F = 1025\,\unit{kg/m^3},\quad \rho_S = 2650\,\unit{kg/m^3},\quad \text{and} \quad \phi = 0.1,
    \\
    \rho = \phi \rho_F + (1 - \phi) \rho_S \quad \text{and} \quad \chi = \rho_F/\phi,
    \\
    \beta = \nu/\kappa,\quad \text{with} \quad \kappa = 10^{-14}\,\unit{m^2},\quad \text{and} \quad\nu = 0.001\,\unit{Pa\cdot s},
    \\
    \alpha = 1\quad \text{and} \quad s = 10^{-9}\,\unit{Pa^{-1}}.
    \end{split}
\end{gather}

\begin{figure}[!ht]
    \begin{center}
        \includegraphics[scale=0.3]{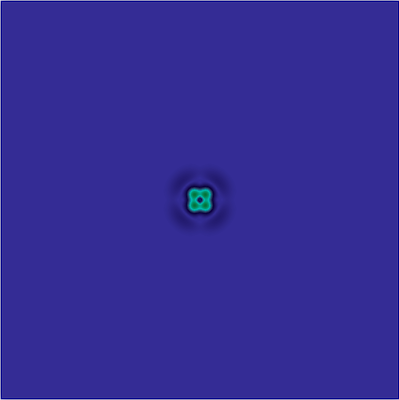}
        \includegraphics[scale=0.3]{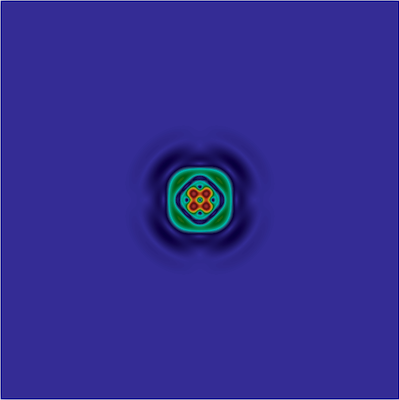}
        \includegraphics[scale=0.3]{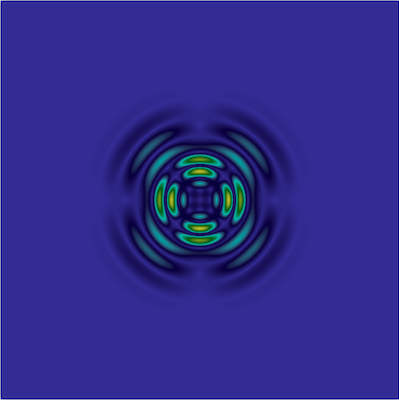}
        \\
        \includegraphics[scale=0.3]{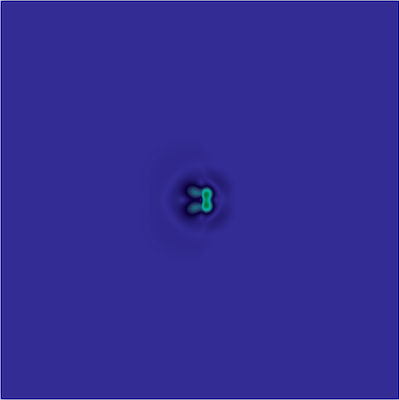}
        \includegraphics[scale=0.3]{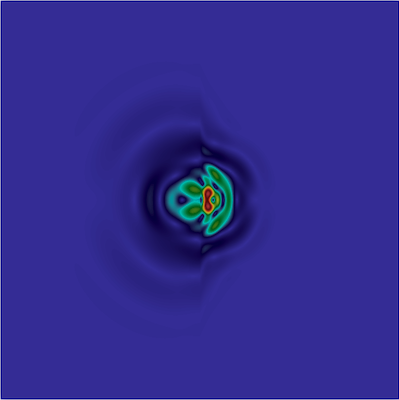}
        \includegraphics[scale=0.3]{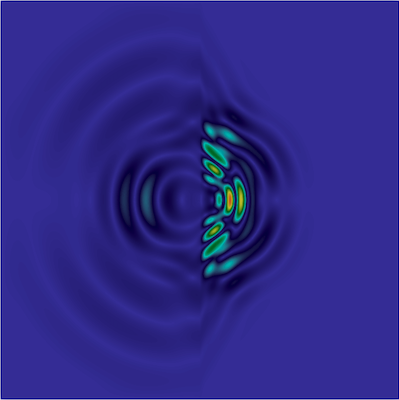}
        \\
    \includegraphics[width=0.5\textwidth]{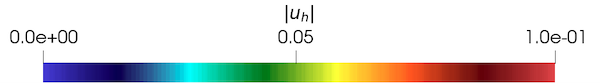}
    \end{center}
    \vspace{-4mm}
    \caption{
        Snapshots of the solid velocity magnitude $\|\bu\|$ at $t = 0.2\,\unit{s}$, $0.4\,\unit{s}$, and $0.6\,\unit{s}$ (left to right). Problem \eqref{eq:poro_thermo_visco} is solved using the source term \eqref{Fshear}, homogeneous initial and Dirichlet boundary conditions, and parameter set \eqref{L4}, with mesh size $h=50$, polynomial degree $k=5$, and time step $\Delta t = 10^{-5}\,\unit{s}$. Top row: purely poroelastic medium ($\omega = 0$ throughout $\Omega$); bottom row: heterogeneous medium with poro-viscoelastic properties ($\omega = 0.9$) in the left half of $\Omega$ and poroelastic properties in the right half.}
    \label{fig:poroElast2}
\end{figure}

Figures \ref{fig:poroElast1} and \ref{fig:poroElast2} present comparative visualizations of pressure and velocity fields in two distinct scenarios: a purely poroelastic medium (top row of each figure) and a heterogeneous medium comprising a poro-viscoelastic region in the left half and a purely poroelastic region in the right half (bottom row). These results effectively demonstrate how the viscoelastic dashpots modify wave propagation characteristics, including wave speeds, front diffusion, amplitude attenuation, and interface phenomena.

In the purely poroelastic case ($\omega=0$ throughout $\Omega$), the punctual shear force generates a characteristic quadrilateral pattern with four distinct lobes of motion in opposing directions (two compressive and two tensile lobes) in both velocity (Figure \ref{fig:poroElast2}) and pressure (Figure \ref{fig:poroElast1}) fields. This cloverleaf-like pattern surrounding the source point is characteristic of shear sources in isotropic media. In this scenario, classical P- and S-waves propagate with attenuation solely due to geometric spreading. Notably, Biot's slow wave is not visible in these simulations due to its significant attenuation by the fluid viscosity $\nu$, rendering it too weak to be observed at this scale, see \cite{meddahi2025}.

In the composite material configuration, with poro-viscoelastic properties ($\omega=0.9$) in the left half and purely poroelastic properties ($\omega=0$) in the right half, we observe the same fundamental wave types but with significantly modified characteristics. The poro-viscoelastic region exhibits waves with reduced amplitude and increased diffusion due to energy dissipation through the Zener dashpot mechanism. Notably, the poro-viscoelastic medium demonstrates higher initial stiffness, resulting in faster wave propagation during early time steps ($t \lesssim \omega$).

\section{Conclusions}
This study has developed a single first-order evolution formulation that captures both Zener-type poro-viscoelasticity and thermo-viscoelasticity, proved its well-posedness with semigroup theory, and designed a hybridizable discontinuous Galerkin discretization that remains energy-stable and hp-optimal on general simplicial meshes. Numerical experiments—from manufactured-solution tests to wave-propagation benchmarks in heterogeneous media—match the predicted convergence rates and reproduce the expected $P$-, $S$-, and $T$-waves, confirming both the analysis and the practical efficiency of the scheme. The framework therefore offers a high-order, robust and computationally economical tool for multiphysics simulations in geomechanics, biomechanics and thermoelastic materials, and it provides a foundation for forthcoming work on adaptive hp-refinement and nonlinear rheologies.

\appendix

\section{Finite element approximation properties}\label{appendix}

This appendix revisits well-established  $hp$--approximation properties, adapting them to match our specific notations and requirements.

For any integer $m\geq 0$ and $K\in \cT_h$, we denote by $\Pi_K^m$ the $L^2(K)$-orthogonal projection onto $\cP_{m}(K)$. This local projection extends naturally to a global projection $\Pi^m_\cT:\, L^2(\cT_h)\to \cP_m(\cT_h)$ by setting $(\Pi^m_\cT v)|_K = \Pi_K^m(v|_K)$ for all $K\in \cT_h$. Similarly, the global projection $\Pi_\cF^m:\, L^2(\cF_h)\to  \cP_{m}(\cF_h)$ is given by $(\Pi^m_\cF \hat v)|_{F} = \Pi_F^m(\hat v|_F)$ for all $F\in \cF_h$, where $\Pi^m_F$ is the $L^2(F)$ orthogonal projection onto $\cP_m(F)$. In the following, we maintain the notation $\Pi_\cT^m$ to refer to the $L^2$-orthogonal projection onto $\cP_m(\cT_h,E)$, for $E\in \{\bbR^d, \bbS\}$. When applied to tensor-valued functions, this projection preserves matrix symmetry as it acts component-wise. Similarly, $\Pi_\cF^m$ also denotes the $L^2$ orthogonal projection onto $\mathcal{P}_m(\mathcal{F}_h, \mathbb{R}^{m\times n})$. For brevity, we write: 
\[
   \Pi_\cT^m (\btau_E, \btau_V, \varphi) = (\Pi_\cT^m \btau_E, \Pi_\cT^m \btau_V, \Pi_\cT^m \varphi) \quad \text{and} \quad \Pi_\cF^m (\bv, \bq) = (\Pi_\cF^m\bv, \Pi_\cF^m\bq).
\]

We start with a fundamental discrete trace inequality.
\begin{lemma}\label{TraceDG}
	There exists a constant $C>0$ independent of $h$ and $k$ such that
\begin{equation}\label{discTrace}
		 \norm{\tfrac{h^{\sfrac{1}{2}}_{\cF}}{k+1} \xi}_{0,\partial \cT_h} \leq C \norm*{ \xi}_{0, \cT_h}\quad \forall  \xi \in \cP_{k}(\cT_h).
\end{equation}
\end{lemma}
\begin{proof} See \cite[Lemma 3.2]{meddahi2023hp}.
\end{proof}

We employ the following two approximation error estimates as key ingredients in deriving our asymptotic $hp$ error bounds.
\begin{lemma}\label{lem:maintool}
	There exists a constant $C>0$ independent of $h$ and $k$  such that
	\begin{equation}\label{tool1}
		\norm{\xi  - \Pi^k_\cT \xi}_{0,\cT_h} + \norm{ \tfrac{h^{\sfrac{1}{2}}_{\cF}}{k+1}(\xi - \Pi_\cT^k \xi)}_{0,\partial \cT_h}
		\leq
		C \tfrac{h_K^{\min\{ r, k \}+1}}{(k+1)^{r+1}}  \norm{\xi}_{1+r,\Omega},
	\end{equation}
	for all $\xi \in  H^{1+r}(\Omega)$, with $r\geq 0$.
\end{lemma}
\begin{proof}
	See \cite[Lemma 3.3]{meddahi2023hp}.
\end{proof}

\begin{lemma}[Projection Error in Product Spaces]\label{maintool2}
There exists a mesh- and polynomial degree-independent constant $C>0$ such that for all $\bw \in H^{2+r}(\Omega,\bbR^{d})$ with $r \geq 0$:
\begin{equation}\label{tool2}
    \abs*{\Big( \bw - \Pi^{k+1}_{\cT} \bw, \bw|_{\partial \cF_h} - \Pi^{k+1}_{\cF} (\bw|_{\partial \cF_h}) \Big) }_{\bigstar}
    \leq
    C \tfrac{h^{\min\{ r, k \}+1}}{(k+1)^{r+\sfrac{1}{2}}}  \norm{\bw}_{2+r,\Omega},
\end{equation}
where $\bigstar$ can be either $\mathcal{Q}$ or $\mathcal{U}$.
\end{lemma}
\begin{proof}
The result follows the proof technique of \cite[Lemma 3.4]{meddahi2023hp}. 
\end{proof}

\bibliographystyle{plainnat}
\bibliography{cvBib.bib}

\end{document}